\documentclass[12pt,reqno]{amsart}
\usepackage{fullpage}
\usepackage{amsfonts}
\usepackage{amssymb}
\usepackage{enumerate}
\usepackage{times}
\usepackage{graphicx}
\usepackage{url}
\usepackage{bm}
\usepackage{eucal}
\usepackage{epstopdf}
\usepackage{mathrsfs}
\usepackage{amsmath}
\usepackage{xcolor}  
\usepackage{mathrsfs}
\usepackage{enumerate}
\usepackage{hyperref}
\usepackage{ccaption}
\usepackage{dsfont}
\usepackage{soul}
\DeclareFontFamily{OT1}{rsfs}{}
\DeclareFontShape{OT1}{rsfs}{n}{it}{<-> rsfs10}{}
\DeclareMathAlphabet{\mathscr}{OT1}{rsfs}{n}{it}
\usepackage{enumitem} 
\usepackage{yfonts}
\usepackage{bbm}

\usepackage{collectbox}

\makeatletter
\newcommand{\mybox}{%
    \collectbox{%
        \setlength{\fboxsep}{1pt}%
        \fbox{\BOXCONTENT}%
    }%
}
\makeatother

\newtheorem{thm}{Theorem}[section]
\newtheorem{conj}[thm]{Conjecture}

\newtheorem{cor}[thm]{Corollary}
\newtheorem{lem}[thm]{Lemma}
\newtheorem{pro}[thm]{Proposition}
\theoremstyle{definition}

\newtheorem{rem}[thm]{Remark}

\numberwithin{equation}{section}

\renewcommand{\mod}[1]{{\ifmmode\text{\rm\ (mod~$#1$)}\else\discretionary{}{}{\hbox{ }}\rm(mod~$#1$)\fi}}
\newcommand{\ep}{\varepsilon}

\newcommand{\N}{{\mathbb N}}

\newcommand{\e}{\varepsilon}

\newcommand{\ex}{\mathbb{E}}

\newcommand{\piup}{\pi}
\newcommand{\mostchisum}{\sum_{\substack{\chi\mod q \\ \chi\ne\chi_0}}}

\newcommand{\Lsum}[1]{\sum_{\substack{#1 \\ L(1/2+i\gamma,\chi)=0}}}

\begin{document}

\title[Prime Number Error Terms]
 {Prime Number Error Terms} 

 \thanks{Nathan Ng was supported by the NSERC discovery grant  RGPIN-2020-06032.  This research was also funded by the PIMS-Collaborative Research Group {\it $L$-functions in Analytic Number Theory}.}

\date{\today}



\author[Nathan Ng]{Nathan Ng}
\address{University of Lethbridge \\ Department of Mathematics and Computer Science \\ 4401 University Drive \\ Lethbridge, AB, Canada \ T1K 3M4 }
\email{nathan.ng@uleth.ca}


\begin{abstract} 
In 1980 Montgomery made a conjecture about the true order of the error term in the prime number theorem. 
In 2012 the author made an analogous conjecture for the true order of the sum of the M\"{o}bius function, $M(x)$. 
This refined an earlier conjecture of Gonek from the 1990's.  In this article we speculate on the true size of a large class of prime number error terms and present a general conjecture. This general conjecture includes both Montgomery's conjecture and the conjecture for $M(x)$  as special cases.    Recently, Lamzouri 
(Springer volume: Essays in Analytic Number Theory, In Honor of Helmut Maier's 70th birthday)
showed that an effective linear independence conjecture (ELI) for the zeros of the zeta function implies one of the inequalities in Montgomery's  conjecture.  
In this article we adapt Lamzouri's method to show that a generalized effective linear independence (GELI) conjecture implies
a lower bound for general prime number error terms.
Furthermore, of independent interest, we prove an $L^2$ bound for almost periodic functions. 
This allows us to weaken significantly one of the conditions in Lamzouri's main result and also give an improvement of the main theorem in an article of Akbary-Ng-Shahabi (Q. J. Math. 65 (2014), no. 3). 
 
\end{abstract}


\thanks{}

\maketitle


\section{Introduction}

The error term in the prime number theorem 
\begin{equation}
 \label{psix}
 \psi(x) = \sum_{p^k \le x} \log p 
\end{equation}
and the sum of the M\"{o}bius function 
\begin{equation}
 \label{Mofx}
 M(x) =  \sum_{n \le x} \mu(n)
\end{equation}
are two of the most studied functions in analytic number theory.  The central reason for this is due to their intimate connection to the Riemann hypothesis \footnote{The Riemann hypothesis shall be denoted RH throughout this article.} and to the distribution of the zeros of the Riemann zeta function.  The Riemann hypothesis is equivalent to the statement
\begin{equation}
 \forall \e >0, \quad
   \frac{\psi(x)-x}{\sqrt{x}} = O(x^{\e})
\end{equation}
and also to 
\begin{equation}
   \forall \e >0, \quad \frac{M(x)}{\sqrt{x}} = O(x^{\e}).
\end{equation}
In 1905, Von Koch \cite{vKo} proved  that RH implies
\begin{equation}
 \label{Vonkoch}
  \frac{\psi(x)-x}{\sqrt{x}} = O(\log^2 x). 
\end{equation}
and in 1914 Littlewood \cite{Lit1} proved that
\begin{equation}
  \label{Littlewood}
   \frac{\psi(x)-x}{\sqrt{x}} = \Omega_{\pm} ( \log \log \log x). 
\end{equation}
Surprisingly, neither of the above bounds has ever been improved.   However,  Schoenfeld established an explicit version 
of \eqref{Vonkoch} and obtained a bound of $\frac{1}{8 \pi} \sqrt{x} \log^2 x$ for $x \ge 73.2$. 
Lee and Nosal \cite{LN} improved this by replacing $\log^2 x$ in Schoenfeld's result with $\log x (\log x -\log \log x)$ for $x \ge 401211790$.  
On the other hand, Heath-Brown \cite{HB} showed that the RH and 
a strong form of the pair correlation conjecture imply that the error term in \eqref{Vonkoch} may be improved to $o( \log^2 x)$. 

One  wonders which estimate \eqref{Vonkoch} or \eqref{Littlewood} lies closer to the truth.  
In \cite{Mo3}, Montgomery made the following remarkable conjecture, which attempts to answer this question.
\begin{conj}[Montgomery, 1980]  \label{Montgomeryconjecture}
\begin{equation}
   \label{limsupinfpsi}
   \limsup_{x \to \infty} \frac{\psi(x) -x}{\sqrt{x} (\log \log \log x)^2} = \frac{1}{2 \pi} \text{ and }
     \liminf_{x \to \infty} \frac{\psi(x) -x}{\sqrt{x} (\log \log \log x)^2} = -\frac{1}{2 \pi}.
\end{equation}
\end{conj}
If one believes this, then  Littlewood's bound \eqref{Littlewood} is closer to the truth.     For $M(x)$, the best known bound on RH is for any $\e >0$, 
\begin{equation}
 \frac{M(x)}{\sqrt{x}} =  O \Big( 
 \exp( \sqrt{\log x} (\log \log x)^{\frac{5}{2}+\e}
 \Big),  
\end{equation}
due to Balazard and de Roton \cite{BdR}.

This improves earlier work of  Soundararajan  \cite{So} and of Maier and Montgomery \cite{MM}.  The best lower bound is due to Greg Hurst \cite{Hu}
\begin{equation}
 \limsup_{x \to \infty} \frac{M(x)}{\sqrt{x}} >  1.826054 \text{ and }
 \liminf_{x \to \infty} \frac{M(x)}{\sqrt{x}} <  -1.837625.
\end{equation}
Note that the results for $M(x)$ are weaker than the corresponding ones for $\psi(x)$.    We can also ask what is the true size of $M(x)$.  In 2012, the author made the following conjecture
\begin{conj}[Ng, 2012]  \label{Ngconjecture}
\begin{equation}
   \label{limsupinfMofx}
   \limsup_{x \to \infty} \frac{M(x)}{\sqrt{x} (\log \log \log x)^{\frac{5}{4}}} = \frac{8a}{5} \text{ and }
     \liminf_{x \to \infty} \frac{M(x)}{\sqrt{x} (\log \log \log x)^{\frac{5}{4}}} = -\frac{8a}{5}
\end{equation}
where
\begin{equation}
  \label{a}
  a = \frac{1}{\sqrt{\pi}} e^{3 \zeta'(-1)-\frac{11}{12} \log 2} \prod_{p} \Big( 
  (1-p^{-1})^{\frac{1}{4}} \sum_{k=0}^{\infty} \Big( \frac{\Gamma(k-\frac{1}{2})}{k! \Gamma(-\frac{1}{2})} \Big)^2  p^{-k}
  \Big). 
  \footnote{By applying the ideas from \cite{Co}, numerical computations show that $a= 0.16712 \ldots $.}
\end{equation}
\end{conj}
This refines a conjecture of Gonek (see \cite[p. 367, eq. (20)]{Ng}) where he had unspecified constants in \eqref{limsupinfMofx}.
Note that there are  other conjectures for the size of $M(x)$ due to Good and Churchhouse  \cite[Conjecture B]{GC}  
and L\'{e}vy (see \cite{Sa}). 
These conjectures are based on modelling $M(x)$ by $\sum_{n \le x} X_n$ where 
the $X_n$ are independent identically distributed random variables wth $\pm 1$ values.   Based on the law of the logarithm 
they conjecture that the maximal size of $M(x)$ is $\pm C \sqrt{x} (\log \log x)^{\frac{1}{2}}$
for certain values of $C$.  One flaw in this model is that it does not take into consideration that the values 
of  $\mu(n)$  are not independent as they depend on the prime factorization of $n$.  
Another model is the Rademacher random multiplicative functions  introduced by Wintner in \cite{Wi2}. 
These models have been extensively studied by Harper \cite{Ha}, \cite{Ha2}. 
There are several other conjectures on the size of $M(x)$.  Kotnik and te Riele have suggested, based on numerical computations, 
that $M(x)$ is of the size $\sqrt{x} (\log \log \log x)^{\Theta}$ with $\Theta =\frac{1}{2}$ and
Kaczorowski has proposed $M(x)$ is this size with $\Theta=1$. The article \cite{KN}
contains a discussion of these various conjectures and suggests that numerical studies may shed light on them. 

In this article we shall
explain where the limits  $\pm \frac{1}{2\pi}$, $\pm \frac{8a}{5}$  in \eqref{limsupinfpsi},  \eqref{limsupinfMofx} and the exponents $2$ and $\frac{5}{4}$ of $\log \log \log x$ in 
\eqref{limsupinfpsi}, \eqref{limsupinfMofx}
arise from. Moreover, we shall provide a more general conjecture which contains
both Montgomery's conjecture and the author's conjecture as special cases.  Furthermore,  we use the argument  from Lamzouri's recent work \cite{La} to prove a conditional result which gives evidence towards the conjecture.
Since the work of Wintner  \cite{Wi}, it has been known that  the distribution of the error term in the prime number theorem is related to the {\it Linear Independence conjecture}
(denoted LI). This is the conjecture
 that the positive ordinates of the zeros of the Riemann zeta function are linearly independent over the rational numbers. 
 A history and discussion of LI  may be found in the introduction in \cite{MN0}. 
 The basic idea is that an error term such as $M(x)$ has an explicit formula of the shape  \cite[Theorem 14.27, p. 372]{Ti}  
 \begin{equation}
  \label{Mxexplicit}
   \frac{M(e^y)}{\sqrt{e^y}} = 2  \sum_{\gamma_n \le e^{y}} \frac{1}{|\rho_n \zeta'(\rho_n)|} 
   \cos \Big(2 \pi i \Big( y \frac{\gamma_{n}}{2 \pi}+\beta_{n} \Big) \Big)
  +o(1)
 \end{equation}
 where $(\gamma_n)_{n \in \mathbb{N}}$ denote the positive ordinates of zeta, 
  $(\beta_n)_{n \in \mathbb{N}}$ is a certain real sequence, and $o(1)$ approaches $0$ as $y \to \infty$. 
 One expects under the assumption of LI that the distribution of values of this sum is the same as the distribution of the random sum 
 \begin{equation}
   \label{Mrandom}
     {\bf X}_{M}  =  2 \sum_{n=1}^{\infty}  \frac{1}{|\rho_n \zeta'(\rho_n)|}  \cos(2 \pi \theta_n)
 \end{equation}
 where  $\theta_n \in [0,1]$ are I.I.D. random variables
 \footnote{Throughout this article I.I.D. is an abbreviation for
 independent identically distributed.}.  
 The LI conjecture can be used to prove results about the limiting distribution $\mu$ of $(\psi(e^y)-e^y)/e^{y/2}$ and related prime counting functions (see \cite{RS}). 
 For instance, it can be used to compute  $\mu(B)$ for certain Borel sets. 
 Montgomery \cite{Mo3} seems to be the first to connect the error term in the prime number theorem to large deviations of sums 
 of independent random variables.  This was also explored in Monach's thesis \cite{Mon}. 
  However, in order to prove omega results concerning prime number error terms it
 seems a stronger version of LI is required.  In  \cite[p.483]{MV} a weaker form of the following conjecture is proposed. 
\begin{conj}[Effective LI conjecture: ELI]\label{ELI}
Let $(\gamma_n)_{n \in \mathbb{N}}$ denote the sequence of  ordinates of the non-trivial zeros of the Riemann zeta function. For every $\ep>0$ there exists a positive constant $C_{\ep}$ (possibly not effective) such that for all real numbers $T\geq 2$ we have 
\begin{equation}\label{ELIsum}\Big|\sum_{0<\gamma_n \leq T} \ell_{n} \gamma_n \Big|\geq C_{\ep} e^{-T^{1+\ep}},
\end{equation}
where the $\ell_{n}$ are integers, not all zero, such that $|\ell_{n}|\leq N(T),$ and $N(T)$ is the number of non-trivial zeros of $\zeta(s)$ with imaginary part in $[0, T]$.
\end{conj}
Damien Roy (personal communication, 2018)
and 
Lamzouri   \cite[Conjecture 1.1]{La0} have provided heuristic arguments towards this conjecture. Further, they have speculated that 
the right hand side of \eqref{ELI} may be replaced by  $\exp(-(\frac{1}{2 \pi} + o(1)) T (\log T)^2)$.
Recently, Lamzouri \cite{La} used the ELI to prove one of the inequalities in Montgomery's Conjecture \ref{Montgomeryconjecture}. 
He showed that ELI implies
\begin{equation}
  \label{limsuppsilb}
     \limsup_{x \to \infty} \frac{\psi(x) -x}{\sqrt{x} (\log \log \log x)^2}  \ge \frac{1}{2 \pi} 
\end{equation}
and that ELI and several conditions on the size of $|\zeta'(\rho_n)|^{-1}$ imply there exists $c_0 >0$ such that 
\[
    \limsup_{x \to \infty} \frac{M(x)}{\sqrt{x} (\log \log \log x)^{\frac{5}{4}}}  \ge c_0. 
\]
Note that in \cite[p.484]{MV} it is claimed that Monach and Montgomery proved a weaker form of ELI (all $\ell_n$ are absolutely bounded)
implies \eqref{limsuppsilb}.  However, the proof does not seem to be in the literature and Lamzouri has doubts whether
this can be done only assuming the weaker form of the ELI conjecture
(see \cite[Remark 1.3]{La}).

Finally, to end the introduction we observe that the aformentioned arguments are not specific to the ordinates of the zeros of the zeta function. 
The ideas and techniques apply to more general sequences.  It is convenient to rephrase the ELI conjecture for arbitrary non-decreasing sequences
of positive numbers. 

\begin{conj}[Generalized Effective LI conjecture: GELI]\label{GELI}
Let $\bm{\lambda} = (\lambda_n)_{n=1}^{\infty}$ be a non-decreasing sequence  of positive numbers which is linearly independent over the rationals. We say that GELI holds for $\bm{\lambda}$ if the following holds: For every $\ep>0$ there exists a positive constant $C_{\ep}$ (possibly not effective) such that for all real numbers $T\geq 2$ we have 
\begin{equation}\label{GELIlbd}\Big|\sum_{0<\lambda_n \leq T} \ell_{n} \lambda_n \Big|\geq C_{\ep} e^{-T^{1+\ep}},
\end{equation}
where the $\ell_{n}$ are integers, not all zero, such that $|\ell_{n}|\leq N_{\bm{\lambda}}(T),$ and 
\begin{equation}
  \label{Nlambda}
N_{\bm{\lambda}}(T)=
\# \{ n \ | \ \lambda_n \le T \}
\end{equation}
is the counting function associated to $\bm{\lambda}= (\lambda_n)_{n=1}^{\infty}$. 
\end{conj}

The reason we consider the effective LI for more general sequences is that we can use it to study error terms associated to the zeros of other $L$-functions.  \\

\noindent {\bf Conventions and Notation}.  
Given two functions $f(x)$ and $g(x)$, we shall interchangeably use the notation  $f(x)=O(g(x))$, $f(x) \ll g(x)$, and $g(x) \gg f(x)$  to mean there exists $M >0$ such that $|f(x)| \le M |g(x)|$ for all sufficiently large $x$. We write $f(x) \asymp g(x)$ to mean that the estimates $f(x) \ll g(x)$ and $g(x) \ll f(x)$ simultaneously hold. We write $f(x) \sim g(x)$ to mean $\lim_{x \to \infty} \frac{f(x)}{g(x)}=1$. We write $f(x)=o(g(x))$ to mean $\lim_{x \to \infty} \frac{f(x)}{g(x)}=0$. 
Given a real random variable  $X$ defined on probability space $(\Omega, \mathcal{B},P)$ we define its expected value by
\[
  \mathbb{E}(X) := \int_{\Omega} X dP
\]
and for a Borel set $B$
\[
   P(X \in B) := P(X^{-1}(B)).
\]

 The zeros of the Riemann zeta function are denoted 
by  $(\rho_n)_{n \in \mathbb{N}}$ where we shall write $\rho_n =\beta_n +i \gamma_n$
and we shall denote $\bm{\gamma}=(\gamma_n)_{n \in \mathbb{N}}$ to be the sequence of its ordinates. 
Throughout this article we will be dealing with two general sequences:
\begin{equation}
   \label{lambdan}
   \bm{\lambda}=(\lambda_n)_{n \in \mathbb{N}} \text{ is a non-decreasing sequence of positive numbers which tend to infinity}
 \end{equation}
 and
\begin{equation}
  \label{rn}
 \mathbf{r}=(r_n)_{n \in \mathbb{N}} \text{ is a complex sequence}.
 \end{equation}
 
 We now introduce the various theorems and conjectures that will be given in this article. 
Following Lamzouri in \cite{La0}, we shall associate to sequences $\bm{\lambda}=(\lambda_n)_{n \in \mathbb{N}}$ and $\mathbf{r}=(r_{n})_{n \in \mathbb{N}}$
a sum of the shape
$$ \Phi_{X,  \bm{\lambda}, \mathbf{r}}(x):= \Re\left(\sum_{0<\lambda_n \leq X} x^{i\lambda_n} r_{n}\right).$$
A key goal is to exhibit large and small values of this function. In order to do so,  we will impose a number of conditions on the sequences 
$\bm{\lambda}=(\lambda_n)_{n \in \mathbb{N}}$ and $\mathbf{r}=(r_{n})_{n \in \mathbb{N}}$. \\

\noindent \textbf{Assumption 1}:  There exist  positive constants $\alpha_{+}, \alpha_{-}, A$ such that 
\begin{equation}
  \label{ass1}
  \alpha_{-}(\log T)^A  \le \sum_{0<\lambda_n \leq T}  |r_{n}| \le \alpha_{+}(\log T)^A
\end{equation}
for $T$ sufficiently large. 

It will be useful to have the following asymptotic version of Assumption 1. \\

\noindent \textbf{Assumption 1B}:  There exist  positive constants $\alpha, A$ such that 
\begin{equation}
  \label{ass1b}
  \sum_{0<\lambda_n \leq T}  |r_{n}| \sim \alpha (\log T)^A
\end{equation}
for $T$ sufficiently large.

  \noindent \textbf{Assumption 2}:
Let $A$ be the constant in Assumption 1. As $T\to \infty$ we have 
\begin{equation}
 \label{ass2}
\sum_{0<\lambda_n \leq T} \lambda_n |r_{n}|=o\big(T(\log T)^A\big).
\end{equation}

 \noindent \textbf{Assumption 3}: There exists a constant  $0 \leq \theta <2$ such that 
\begin{equation}
  \label{ass3}
\sum_{0<\lambda_n\leq T} \lambda_n^2 |r_{n}|^2 \ll T^{\theta}.
\end{equation}

\noindent \textbf{Assumption 4}: 
There exist positive constants $\kappa_{-}, \kappa_{+}$ such that for
$T \ge 1$, we have the bounds
\begin{equation}
  \label{ass4}
   \kappa_{-} T (\log T) \le  N_{ \bm{\lambda}}(T)  \le  \kappa_{+} T (\log T). 
\end{equation}

\noindent \textbf{Assumption 5}: For $T \ge 0$, 
\begin{equation}
\label{ass5}
 \sum_{T < \lambda_n \le T+1} 1 \ll \log (T+2)
 \end{equation}
 
 For context, we explain what each of these assumptions mean in the 
 special case of $M(x)$ where $r_n=\frac{1}{\rho_n \zeta'(\rho_n)}$.  Assumptions 1 and 1B deal with bounds and asymptotics for $\sum_{0 < \gamma_n \le T} \frac{1}{|\rho_n \zeta'(\rho_n)|}$.   Assumption 2 and 3 will correspond to bounds for $J_{-1/2}(T)$ and $J_{-1}(T)$ (see \eqref{JkT} below for definitions).  Assumptions 4 and 5 correspond to bounds for the number of ordinates $\gamma_n$ in a long and in a short interval.  A key point of this article is to apply arguments that apply to any sequence and it is useful to keep this example in mind.  Note that the classical prime number example $\psi(x)$ is less relevant to the general case since the corresponding coefficients $|r_n| = |\rho_n|^{-1}$ decrease and they are easier to treat.  
 
 With the above assumptions in hand, we can state our main theorem which produces large and small values of $ \Phi_{X^2,  \bm{\lambda}, \mathbf{r}}(x)$.
  \begin{thm}\label{generalsumthm}
Assume $\bm{\lambda}= (\lambda_n)_{n \in \mathbb{N}} $ satisfies the Generalized Effective LI conjecture (Conjecture \ref{GELI}).
Assume the sequences $\bm{\lambda}= (\lambda_n)_{n \in \mathbb{N}} $  and  $ \mathbf{r}=(r_n)_{n \in \mathbb{N}}$
  satisfy Assumptions 1-5, \eqref{ass1},\eqref{ass2}, \eqref{ass3}, \eqref{ass4}, \eqref{ass5}.
 Let $X$ be large and let  $\e \in (0,\alpha_{-})$ where $\alpha_{-}$ is the constant  in \eqref{ass1}.  Then
\begin{equation}\label{MaxOmega}\max_{x\in [2, X]}
 \Phi_{X^2,  \bm{\lambda}, \mathbf{r}}(x)
 \geq (\alpha_{-}-\e) (\log\log\log X)^A
\end{equation}
and 
\begin{equation}\label{MinOmega}\min_{x\in [2, X]} \Phi_{X^2,  \bm{\lambda}, \mathbf{r}}(x) \leq -(\alpha_{-}-\e) (\log\log\log X)^A.
\end{equation}
 \end{thm}
 \begin{rem}
 The main ideas for this result are due to Lamzouri \cite{La0}. 
 This theorem is a modification of Lamzouri's Theorem 1.5 of \cite{La0} which deals with the sequence of zeta zero ordinates
 $\bm{\gamma}=(\gamma_n)_{n \in \mathbb{N}}$.
 The main differences are that this theorem is stated for general sequences $\bm{\lambda}= (\lambda_n)_{n \in \mathbb{N}} $ 
 and  the constant $c_1$ in the cited theorem is given to be $\alpha_{-}$ where $\alpha_{-}$ is the lower bound constant in 
 Assumption 1 \eqref{ass1}. 
 This was determined by examining closely the proofs of Proposition 2.1, Proposition 3.1, and Theorem 1.5 of  \cite{La0}.   The key point is to use a sharper lower bound for the  modified Bessel function $I_0$.
 Another new input in this result is Assumption 3 \eqref{ass3} above. Previously,  in \cite{La0} the stronger
 condition  $\theta < 3-\sqrt{3} =1.26 \ldots$ was assumed. This value comes from the proof of 
  Theorem 1.2 (b) of \cite{ANS}.
  The weakening of this condition is achieved by using an idea of Meng \cite{Me}
 in his study of the distribution $k$-free numbers and makes use of Proposition \ref{MengProp} below.  
\end{rem}
 
 We now discuss the applications of this conjecture to extreme values of $M(x)$ and $L(x)=\sum_{n \le x} \lambda(n)$
 where $\lambda(n)$ is the Liouville function.  From their explicit formulae (see \eqref{Mxexplicit}, \eqref{Lxexplicit}), these functions are intimately related to the numbers $\zeta'(\rho_n)^{-1}$.  It is convenient to define for $k \in \mathbb{R}$, 
 \begin{equation}
 \label{JkT}
 J_k(T) = \sum_{0 < \gamma_n < T}
 |\zeta'(\rho_n)|^{2k}.
 \end{equation} 
 Hughes, Keating, and O'Connell \cite{HKO}  made the following conjecture. 
\begin{conj}[Hughes-Keating-O'Connell] \label{HKO}
Let $k \in \mathbb{C}$ with  $\Re(k) > -\frac{3}{2}$. Then 
\begin{equation}
  \label{HKOconj}
  J_{k}(T) = \sum_{0 < \gamma_n < T} |\zeta'(\rho_n)|^{2k} 
  \sim 
   \frac{G^2(2+k)}{G(3+2k)}  a(k)
  \frac{T}{2 \pi}   \Big(\log  \frac{T}{2 \pi} \Big)^{(k+1)^2} \, \, \, \, 
\end{equation}
where $G(\cdot)$ is the Barnes $G$-function defined by 
\[
   G(z+1) =(2 \pi)^{z/2}\exp \left(-\frac{1}{2}(z^{2} + \gamma z^{2} + z)
   \right)
   \prod_{n=1}^{\infty} \Big(1 + \frac{z}{n} \Big)^{n} e^{-z + z^{2}/2n}
\] 
and 
\begin{equation}
  \label{adefn}
   a(z) = \prod_{p} \Big(1-\frac{1}{p} \Big)^{x^2} \sum_{m=0}^{\infty} \Big( \frac{\Gamma(m+z)}{m! \Gamma(z)} \Big)^2 p^{-m}.
\end{equation}
\end{conj}
Note that the conjecture in the $k=1$ case is 
\begin{equation}
  \label{Jm1T}
    \sum_{0 < \gamma_n < T} \frac{1}{|\zeta'(\rho_n)|} 
  \sim a T (\log T)^{\frac{1}{4}}
\end{equation}
where $a$ is the constant given in \eqref{a}. 
Conjecture \ref{HKOconj} is a refinement of an earlier 1989 conjecture of Gonek \cite{Go} and Hejhal \cite{He}
which asserts that 
\begin{equation}
 \label{GHconjecture} 
 \text{for } k > - \frac{3}{2}, \, \, 
   \sum_{0 < \gamma_n < T} |\zeta'(\rho_n)|^{2k}  \asymp T (\log T)^{(k+1 )^2}.
  \, \, 
   \footnote{The original conjecture of Gonek and Hejhal was for all real $k$, but is now believed be false for $k \le -\frac{3}{2}$ (see \cite[pp.2616-2617]{HKO}). }
\end{equation} 
Conjecture \ref{HKO} is only known to be true for $k=0$ (Riemann-von-Mangoldt) and $k=1$  (Gonek \cite{G1}, assuming the Riemann Hypothesis). 
There is more evidence supporting the truth of the weaker  Gonek-Hejhal conjecture. 
In particular, assuming RH, combining the work of Kirila \cite{Ki}  and Benli-Elma-Ng \cite{BEN} establishes \eqref{GHconjecture} for integers $k \ge 0$. Kirila's work also establishes upper bounds of the correct order of magnitude for real values of $k \ge 0$. 
For negative moments, Gonek \cite{Go} gave a lower bound in the case $k=-1$ and Milinovich and Ng made this explicit \cite{MN}. 
For fractional $k<0$,  Heap-Li-Zhao \cite{HLZ}, have established the lower bound in \eqref{GHconjecture} assuming RH and the simplicity of zeta zeros. Gao-Zhao \cite{GZ} have
extended this to all real $k <0$.  Currently no one has proven upper bounds for $J_k(T)$, in the case $k <0$. 

 From Theorem \ref{generalsumthm} we have the following Corollary on $M(x)$. 
 \begin{cor}\label{Mofxcorr}
Assume the Effective LI conjecture (Conjecture \ref{ELI}).  Assume there exist positive constants $a_{-}, a_{+}$ such that 
for $T$ sufficiently large
\begin{equation}
   \label{Jminus1over2T}
  a_{-}T (\log T)^{\frac{1}{4}} \le    \sum_{0 < \gamma_n < T}  \frac{1}{|\zeta'(\rho_n)|}   \le  a_{+} T (\log T)^{\frac{1}{4}}
\end{equation}
and assume that for $T$ sufficiently large
\begin{equation}
    \label{Jminus1T}
     \sum_{0 < \gamma < T}  \frac{1}{|\zeta'(\rho_n)|^2}  \ll T^{2-\e}
\end{equation}
where $\e$ is sufficiently small. 
 Then we have
\begin{equation}\label{LimsupMofx}
 \limsup_{x\to \infty} \frac{M(x)}{\sqrt{x}(\log\log\log x)^{\frac{5}{4}}}\geq  \frac{8a_{-}}{5}
 \end{equation}
 and 
 \begin{equation}\label{LiminfMofx}
 \liminf_{x\to \infty} \frac{M(x)}{\sqrt{x}(\log\log\log x)^{\frac{5}{4}}}\leq - \frac{8a_{-}}{5}.
 \end{equation}
\end{cor}
\begin{rem}
Assuming RH and the simplicity of zeros of $\zeta(s)$, 
the lower bound for $J_{-1/2}(T)$ in \eqref{Jminus1over2T} is known to be true  for some $a_{-}$ due to the work of Heap, Li, and Zhao \cite{HLZ}.
However, an explicit value of $a_{-}$ has not been computed.   There are no known upper
bounds proven for $J_{-1}(T)$.  In fact, it is known that  $J_{-1}(T)=o(T^2)$ assuming the weak Mertens conjecture:
\begin{equation}
  \int_{1}^{X} \Big( \frac{M(x)}{x} \Big)^2 \, dx \ll \log X. 
\end{equation}
Recently, Bui, Florea, and Milinovich \cite{BFM} have shown that on RH a variant  of the sum $J_{-1}(T)$  is $O(T^{\frac{3}{2}})$ where zeros such that $|\zeta'(\rho_n)|$ is extremely small are excluded from the sum. 
\end{rem}

We have an analogous corollary for $L(x)$.  Note that $L(x)$ has an explicit formula 
(see \cite{F} and \cite[p. 773, eq. (4.21)]{ANS})
of the shape
 \begin{equation}
  \label{Lxexplicit}
   \frac{L(e^y)}{\sqrt{e^y}} = 2  \sum_{\gamma_n \le e^{y}} 
   \Big| \frac{\zeta(2\rho_n)}{\rho_n \zeta'(\rho_n)} \Big|
   \cos \Big(2 \pi i \Big( y \frac{\gamma_{n}}{2 \pi}+\beta_{n}' \Big) \Big)
  +o(1)
 \end{equation}
 where $(\gamma_n)_{n \in \mathbb{N}}$ denote the positive ordinates of zeta, 
  $(\beta_n')_{n \in \mathbb{N}}$ is a certain real sequence, and $o(1)$ approaches $0$ as $y \to \infty$. 
In this case the numbers $\frac{\zeta(2\rho_n)}{\rho_n \zeta'(\rho_n)}$ shall determine the size of $L(x)$. 
\begin{cor}\label{Lofxcorr}
Assume the Effective LI conjecture (Conjecture \ref{ELI}).  Assume there exist positive constants $b_{-}, b_{+}$ such that 
for $T$ sufficiently large

\begin{equation}
    \label{Kminus1over2T}
   b_{-}T (\log T)^{\frac{1}{4}} \le   
  \sum_{0 < \gamma < T}  \Big| \frac{\zeta(2 \rho_n)}{\zeta'(\rho_n)} \Big| 
  \le  b_{+}T (\log T)^{\frac{1}{4}} 
\end{equation}
and assume that for $T$ sufficiently large
\begin{equation}
    \label{Kminus1T}
     \sum_{0 < \gamma < T}  \frac{1}{|\zeta'(\rho_n)|^2}  \ll T^{2-\e}
\end{equation}
where $\e$ is sufficiently small.  Then we have
\begin{equation}\label{LimsupLofx}
 \limsup_{x\to \infty} \frac{L(x)}{\sqrt{x}(\log\log\log x)^{\frac{5}{4}}}\geq  \frac{8b_{-}}{5}
 \end{equation}
 and 
 \begin{equation}\label{LiminfLofx}
 \liminf_{x\to \infty} \frac{L(x)}{\sqrt{x}(\log\log\log x)^{\frac{5}{4}}}\leq - \frac{8b_{-}}{5}.
 \end{equation}
\end{cor}
We now address the possible values for the constants $a_{\pm}$ and $b_{\pm}$. 
In light of the conjectured asymptotic \eqref{Jm1T}, we expect that $a_{+} =a+\e$ and $a_{-}=a-\e$, where $\e$ can be taken to be any sufficiently small positive constant. 
Based on Conjecture \ref{HKO}  it seems natural to make the following conjecture.
\begin{conj} \label{ANY2}
  If $0 \le s < 3$, 
then 
\begin{equation}
  \label{asymp}
    \sum_{0 \le \gamma_n\leq T}
\Big|\dfrac{\zeta(2\rho_n)}{\zeta'(\rho_n)}\Big|^{s}    \sim  \frac{G^2(2-\tfrac{s}{s})}{G(3-s)} a(-\tfrac{s}{2}) \Big( \sum_{n=1}^{\infty} \frac{d_{s/2}(n)}{n^2}  \Big) \frac{1}{2 \pi}
   \int_{1}^{T}  \Big(\log \frac{t}{2 \pi} \Big)^{(s/2-1)^2} dt
\end{equation}
where $d_k(\cdot)$ denotes the $k$-th divisor function \footnote{ $d_k(n)$ is defined to be the $n$-th coefficient of the Dirichlet series of $\zeta(s)^k$: i.e. $\zeta(s)^k = \sum_{n=1}^{\infty} d_k(n) n^{-s}$.}. 
\end{conj}
This conjecture arose from joint work with summer student Yang Li, Amir Akbary, and Majid Shahabi.   One can try to use Conjecture \ref{HKO} and partial summation to prove the asymptotic in \eqref{asymp}, but it seems tricky to bound the error term. 
Note that Conjecture \ref{ANY2} implies 
\begin{equation}
  \label{bvalue}
   \sum_{0 \le \gamma_n\leq T}
\Big|\dfrac{\zeta(2\rho_n)}{\zeta'(\rho_n)}\Big| \sim b T(\log T)^{\frac{1}{4}} 
\text{ where } b = a   \cdot  \Big( \sum_{n=1}^{\infty} \frac{d_{1/2}(n)}{n^2}  \Big). 
\end{equation} 
This leads to the following conjecture for the extremal values of $L(x)$. 
\begin{conj}[2012]  \label{Lofxconjecture}
\begin{equation}
   \label{limsupinfLofx}
   \limsup_{x \to \infty} \frac{L(x)}{\sqrt{x} (\log \log \log x)^{\frac{5}{4}}} = \frac{8b}{5} \text{ and }
     \liminf_{x \to \infty} \frac{L(x)}{\sqrt{x} (\log \log \log x)^{\frac{5}{4}}} = -\frac{8b}{5}
\end{equation}
where $b$ is given in \eqref{bvalue}.
 
\end{conj}
Conjectures \ref{ANY2} and \ref{Lofxconjecture} first appear in the article \cite{ANS2}.
Observe that the $s=2$ case of Conjecture \ref{ANY2}  reduces to
\begin{equation}
    \sum_{0 \le \gamma_n\leq T}
\Big|\dfrac{\zeta(2\rho_n)}{\zeta'(\rho_n)}\Big|^2 \sim \frac{T}{2 \pi}. 
\end{equation}
This was previously conjectured in the author's Ph.D. thesis \cite[p.147]{Ng} by using a different method.  The author derived the conjecture by following ideas from  an unpublished article of Gonek (see \cite{Go2}), which are closely related to Montgomery's work on the pair correlation conjecture \cite{Mo1}.

We now discuss prime number error terms associated to real valued functions $\varphi$
which possess an explicit formula. 
Let $\varphi: [0,\infty)\to\mathbb{R}$
and let $y_0$ be a non-negative constant such that $\varphi$ is square-integrable on $[0, y_0]$. We shall assume 
there exists  $c$ a real constant, sequences
 $\mathbf{r}=(r_n)_{n \in \mathbb{N}}$,  $ \bm{\lambda}=(\lambda_n)_{n \in \mathbb{N}}$  of the types, \eqref{rn}, \eqref{lambdan},  which satisfy 
 Assumption 4 \eqref{ass4}  such that for $y \ge y_0$
 \begin{equation} \label{phiexplicit}
\varphi(y) = c + 2\Re \Big( \sum_{\lambda_n \le X}r_ne^{i\lambda_ny} \Big) + \mathcal{E}(y,X) ,
\footnote{In \cite{ANS} the definition of $\varphi(y)$ does not have the factor of 2 in front of the sum.  This is included for
consistency with the rest of this article. }
\end{equation}
for any $X \ge X_0>0$
where $\mathcal{E}(y,X)$ satisfies 
\begin{equation} \label{EyX}
\lim_{Y\to\infty}\dfrac{1}{Y}\int_{y_0}^Y|\mathcal{E}(y,e^Y)|^2dy=0.
\end{equation}
The prime number error term associated to $\varphi$ is given 
by 
\begin{equation}
  \label{Evarphix}
 E_{\varphi}(x) = \varphi(\log x).
\end{equation}
With this change of variable, the explicit formula is 
\begin{equation}
  \label{Eexplicit}
  E_{\varphi}(x)= c + 2\Re \Big( \sum_{\lambda_n \le X}r_n x^{i\lambda_n} \Big) + 
  \mathcal{F}(x,X)
\end{equation}
where $ \mathcal{F}(x,X):=
  \mathcal{E}(\log x,X)$
and we assume it satisfies the bound
\begin{equation}
  \label{errorbd}
  \mathcal{F}(x,X^2)=  \mathcal{E}(\log x,X^2) \ll 1 \, \text{ for }  \, 2 \le x \le X. 
\end{equation}
We present a general conjecture on the prime number error terms $E_{\varphi}(x)$. 
A version of this first appears in the unpublished manuscript \cite{ANS2}. 
\begin{conj} \label{phiconj}
Let $\varphi$ be a function of the above type which satisfies \eqref{phiexplicit} and \eqref{EyX}
and $E_{\varphi}(x)$ is the associated error term as in \eqref{Evarphix} which satisfies
\eqref{errorbd}. 
Assume  that the sequences 
$\bm{\lambda}=(\lambda_n)_{n \in \mathbb{N}}$ and $\mathbf{r}=(r_{n})_{n \in \mathbb{N}}$
  satisfy Assumption 3 \eqref{ass3} for some $\theta$ with $0 < \theta <2 $ 
  and assume  $\bm{\lambda}=(\lambda_n)_{n \in \N}$ satisfies  Assumption 4 \eqref{ass4}.
  If  $\bm{\lambda}=(\lambda_n)_{n \in \mathbb{N}}$ and $\mathbf{r}=(r_{n})_{n \in \mathbb{N}}$
  satisfy Assumption 1B: 
  \begin{equation}
   \sum_{0<\lambda_n \leq T}  |r_{n}| \sim \alpha (\log T)^{A} , 
  \end{equation}
 then we have 
\begin{equation*}
     \limsup_{x \to \infty} \frac{E_{\varphi}(x)}{ (\log \log \log x)^{A}} =  2\alpha
     \text{ and }
       \liminf_{x \to \infty} \frac{E_{\varphi}(x)}{ (\log \log \log x)^{A}} = -2\alpha.
\end{equation*}
\end{conj}
Note that this general conjecture contains Conjectures \ref{Montgomeryconjecture}, \ref{Ngconjecture}, and \ref{Lofxconjecture} as
 special cases.  The following table summarizes the values of $r_n$, $\lambda_n$, and other related quantities. 
\begin{center}
\begin{table}[h!]
\caption{Special cases of Conjecture \ref{phiconj}}
\begin{tabular}{|c|c|c|c|c|c|c|}
 \hline
   $\varphi(y)$ & $E_{\varphi}(x)$ & $r_{n}$ &  $\lambda_n$ & $\sum_{0 < \lambda_n < T}  |r_{n}|$  & $2\alpha$ & $A$ \\
 \hline
 \hline
   $\tfrac{\psi(e^y)-e^y}{e^{y/2}}$ &  $\tfrac{\psi(x)}{\sqrt{x}}$ & $\tfrac{1}{\rho_n}$ &  $\gamma_n$ & $\frac{1}{4 \pi} (\log T)^2$ & $\frac{1}{2 \pi}$ & 2 \\
 \hline
  $\tfrac{M(e^y)}{e^{y/2}}$ &  $\tfrac{M(x)}{\sqrt{x}}$ & $\tfrac{1}{\rho_n \zeta'(\rho_n)}$ & $\gamma_n$ &$\frac{4a}{5} (\log T)^{\frac{5}{4}}$ & $\frac{8a}{5}$ & $\frac{5}{4}$ \\
  \hline
  $\tfrac{L(e^y)}{e^{y/2}}$ & $\tfrac{L(x)}{\sqrt{x}}$ &  $\tfrac{\zeta(2 \rho_n)}{\rho_n \zeta'(\rho_n)}$ & $\gamma_n$ & $\frac{4b}{5} (\log T)^{\frac{5}{4}}$ & $\frac{8b}{5}$ & $\frac{5}{4}$  \\
  \hline
\end{tabular}
\end{table}
\end{center}

We are able to prove that GELI implies one side of the inequalities in the above conjecture. 
\begin{thm}
\label{conjthm}
Assume the Generalized Effective LI conjecture (Conjecture \eqref{GELI}).
Let $E_{\varphi}(x)$ be a function which satisfies an explicit formula of the type \eqref{Evarphix}  and assume that 
we have the estimate \eqref{errorbd}.
Also assume that the sequences 
$\bm{\lambda}=(\lambda_n)_{n \in \mathbb{N}}$ and $\mathbf{r}=(r_{n})_{n \in \mathbb{N}}$ satisfy Assumptions 1B, 2-5
\eqref{ass1b}, \eqref{ass2}, \eqref{ass3}, \eqref{ass4}, \eqref{ass5}.   Then we have 
\begin{equation}
  \label{lblimsup}
      \limsup_{x \to \infty} \frac{E_{\varphi}(x)}{ (\log \log \log x)^{A}} \ge   2 \alpha
      \text{ and }
        \liminf_{x \to \infty} \frac{E_{\varphi}(x)}{ (\log \log \log x)^{A}} \le   -2 \alpha.
\end{equation}
\end{thm}
This result provides strong evidence toward Conjecture \ref{phiconj}.  
We now provide  some heuristic reasoning behind Conjecture \ref{phiconj} and why we believe
the above inequalities should be sharp.   
Note that the function $\varphi(y)$ can be modelled by the random sum
\begin{equation}
  \label{Xr}
  {\bf X}_{\bf r} = 2 \sum_{n=1}^{\infty} r_{n} \cos(2 \pi \theta_n) 
\end{equation}
where ${\bf r}= (r_{n})_{n \in \mathbb{N}}$ and
$\theta_n \in [0,1]$ are I.I.D. random variables.   
 The following large deviation result was established in \cite{ANS2}.
\begin{thm}[Akbary, Ng, Shahabi, 2012, unpublished] \label{ANSunpub}
Let $\e >0$.   Let  ${\bf X}_{\bf r}$ be the random variable \label{Xr} where  the sequences 
$ \mathbf{r}=(r_n)_{n \in \mathbb{N}}$
 and $\bm{\lambda}= (\lambda_n)_{n \in \mathbb{N}} $
 satisfy Assumption 1B \eqref{ass1b}, Assumption 5 \eqref{ass4}, and the following three 
 assumptions:
 \begin{enumerate}
 \item[(i)] There exists a positive $c$ such that  $\sum_{\lambda_n > T} |r_n|^2 \ll   \frac{(\log T)^{c}}{T}$.
 \item[(ii)] For every $0< \varepsilon < 1$, 
 $  \sum_{n=1}^\infty |r_n|^{1+\varepsilon}  \ll_{\e} 1$. 
 \item[(iii)]   For every $\varepsilon >0$, 
 there exists $C:=C(\e)>0$ such that 
$|r_n| \ge  \frac{C}{\lambda_n^{1+\varepsilon}}$, for all $n \in \mathbb{N}$. 
   \end{enumerate}
Then for $V \ge V_{\e}$, we have
\begin{equation}
 \label{largedevXphi}
\exp\bigg(-\exp\Big(\big(\alpha^{\frac{1}{A}}+\varepsilon\big)V^{\frac{1}{A}}\Big)\bigg)
\leq P\big( {\bf X}_{\bf r}  \geq V\big) \le 
\exp\bigg(-\exp\Big(\big(\alpha^{\frac{1}{A}}-\varepsilon\big)V^{\frac{1}{A}}\Big)\bigg).
\end{equation}
\end{thm}
This result shows that the upper bound in \eqref{largedevXphi} is essentially sharp. The constants in the upper and lower bounds differ are differ by an $\e$. 
The upper bound in this result follows \cite[Theorem 2]{MO} which is  an application of Chernoff's inequality.
The lower bound can be established using ideas from \cite[Sec. 3, Theorem 1]{Mo3}. 
Note that there are sharper results in the literature.  The work of Granville and Lamzouri \cite{GL} derive more precise results for 
$P\big( {\bf X}_{\bf r}  \geq V\big)$  in the case that $ \mathbf{r}=(r_n)_{n \in \mathbb{N}}$ is decreasing. 
See also  Majid Shahabi's 2012 M.Sc. thesis \cite{Sh} where related results are proven. In Shahabi's work the coefficients 
$ \mathbf{r}=(r_n)_{n \in \mathbb{N}}$ are not assumed to be decreasing.  It is should be emphasized that in the cases of the distribution of  $M(x)$ and $L(x)$ the corresponding  coefficients $ \mathbf{r}=(r_n)_{n \in \mathbb{N}}$ do not decrease. 
In the special case of the random variable ${\bf X}_M$ \eqref{Mrandom} bounds of the above type were derived in \cite[Corollary 12, pp.384-387]{Ng}.
At the time, the author was unable to find the same constant multiplying the $V^{1/A}$ term in the exponent.   That is the constants $\tilde{c}_1, \tilde{c}_2$ in the cited article 
were off by a constant factor. 

Under suitable conditions the function $\varphi(y)$ (see Theorem \ref{ANScor} below) possesses a limiting distribution (probability measure) $\mu$. Namely, 
\[
  \lim_{Y \to \infty} \frac{1}{Y} \int_{1}^{Y} f(\varphi(y)) \, dy  = \int_{-\infty}^{\infty} f(t) d \mu(t)
\]
for all bounded, continuous functions $f(x)$ on $\mathbb{R}$. 
In the case that $\mu$ is absolutely continous and the  $\bm{\lambda}= (\lambda_n)_{n \in \mathbb{N}} $ 
are linearly independent over the rationals, we have 
\begin{equation}
  \label{limdisac}
  \lim_{Y \to \infty} \frac{1}{Y} 
  \text{meas} \Big( y \in [1,Y] \ | \ \varphi(y) \ge V \Big)
  =  \mu([V, \infty)) = P\big( {\bf X}_{\bf r}  \geq V\big). 
\end{equation}
Let us assume that $P\big( {\bf X}_{\bf r}  \geq V\big) \le  \exp(-\exp (CV^{\frac{1}{A}}))$ for the positive constant $C= (\alpha^{\frac{1}{A}}-\varepsilon)$. Now consider the event 
\[
  E_n = \Big\{ \underline{\theta} \in \mathbb{T}^{\infty} \ | \ 
      {\bf X}_{\bf r}(\underline{\theta}) \ge \left( \frac{1}{C}
             \log \log (n (\log n)^{\kappa}) \right)^{A} 
   \Big\} 
\]
where $\kappa > 1$. 
Therefore we have by the upper bound on $P\big( {\bf X}_{\bf r}  \geq V\big) $
\begin{equation}
   \sum_{n=n_{0}}^{\infty} P(E_{n}) 
   \ll  \sum_{n=n_{0}}^{\infty} \frac{1}{n(\log n)^{\kappa}}  
   \ll 1 
\end{equation}
for $n_{0}$ a sufficiently large integer. The  Borel-Cantelli lemma implies that 
\begin{equation}
  P(E_{n} \ \text{ infinitely often}) = 0. 
\end{equation}
Therefore if the convergence of \eqref{limdisac}  is
sufficiently uniform then we expect that
\begin{equation}
        \limsup_{y \to \infty} \frac{\varphi(y) }{ (\log \log
  y)^{A}}
  \le \left( \frac{1}{C} \right)^{A} = 
  \left( \frac{1}{\alpha^{\frac{1}{A}}-\varepsilon} \right)^{A} =\alpha + o(1).
\end{equation}
Since $E_{\varphi}(x) = \varphi(\log x)$ it follows that 
\begin{equation}
     \limsup_{x \to \infty} \frac{E_{\varphi}(x) }{ (\log \log
  \log x)^{A}}
  \le \alpha + o(1).
\end{equation}
Note that we do not expect to find a smaller bound since Theorem \ref{ANSunpub} 
shows that the upper bound for large deviations  is essentially sharp. 
This last inequality combined with the first inequality in \eqref{lblimsup} leads to Conjecture \ref{phiconj}.

To end this introduction, we explain how the condition $0 \le \theta < 2$ in Assumption 3 \ref{ass3} arises and 
we present an independent result. 
The following Proposition improves \cite[Theorem 1.2, pp.19-22]{ANS}
which in turn was a generalization of  \cite[Lemma 6, p. 372]{Ng} and an argument from \cite[pp.148-151]{Cr}. 
\begin{pro} \label{MengProp}
Assume  that the sequences 
$\bm{\lambda}=(\lambda_n)_{n \in \mathbb{N}}$ and $\mathbf{r}=(r_{n})_{n \in \mathbb{N}}$
  satisfy  Assumption 3 \eqref{ass3} for some $\theta$ with $0 < \theta <2 $
and assume  $\bm{\lambda}=(\lambda_n)_{n \in \N}$ satisfies  Assumption 5 \eqref{ass5}.
Let $V$ be a real number and $X > T \ge 1$. 
Then for $\epsilon>0$ sufficiently small
\begin{equation}
   \label{Vintbd}
    \int_{V}^{V+1}\Big|\sum_{T<\lambda_n\leq X}r_ne^{iy\lambda_n}\Big|^2dy 
   \ll 
  \frac{1}{T^{2-\theta-\epsilon}}
\end{equation}
and
\begin{equation}
 \label{rnrmbound}
\sum_{T<\lambda_n\leq X}\sum_{T<\lambda_m\leq X}|r_nr_m|\min\left(1,\dfrac{1}{|\lambda_n-\lambda_m|}\right)
\ll T^{-(2-\theta-\epsilon)}. 
\end{equation}
\end{pro}
This Proposition is a formalization of an argument of Meng \cite[Lemma 5, p.298]{Me}. 
Meng considers the special case 
\[
 r_n = \frac{\zeta(\rho_n/k)}{\zeta'(\rho_n)} \text{ and }
 \lambda_n =\gamma_n \text{ where }  k \in \mathbb{Z}_{\ge 2}
 \]
 in his study of the  distribution of $k$-free numbers. 
As a consequence of this proposition,   we also establish a limiting distribution result which improves Corollary 1.3 [b] in \cite{ANS}.

\begin{thm}\label{ANScor}
Let $\varphi: [0,\infty) \to \mathbb{R}$ satisfy \eqref{phiexplicit} and \eqref{EyX}.  
Assume $(\lambda_n)_{n \in \N}$ satisfies Assumption 5 \eqref{ass5}
 and 
assume there exists  $0\leq \theta< 2$ such that the sequences 
$ \mathbf{r}=(r_n)_{n \in \mathbb{N}}$
 and $\bm{\lambda}= (\lambda_n)_{n \in \mathbb{N}} $
  satisfy Assumption 3 \eqref{ass3}.
Then $\varphi(y)$ is a $B^2$-almost periodic function and therefore possesses a limiting distribution.
\end{thm}
This improves Corollary 1.3 (b) of \cite{ANS} since in that article the condition \eqref{ass3} was only valid for $\theta < 3- \sqrt{3}$. 
Observe that the proof of Theorem \ref{ANScor} is exactly the proof of Corollary 1.3 (b) of \cite{ANS}.  The only difference is that the bound (1.13) \cite[p. 4]{ANS} has the condition $\theta <  3- \sqrt{3}$ replaced by $\theta <2$. 
As noted in the proof of Corollary 1.3 (b) \cite[p. 23]{ANS} we are actually using the fact that $\theta <2$ and the proof works verbatim.  This general result is useful for proving  limiting distribution results for number theoretic error terms, including those in the  prime number race game. 
 For instance, Theorem \ref{ANScor} can be used to weaken one of the conditions in Corollary 1.6 of \cite{ANS}.


\section{Key ingredients and outline of the Proof of Theorem \ref{generalsumthm}}

We now give an overview of Lamzouri's argument which establishes large values, assuming Conjecture \ref{GELI} (GELI). 
We define 
\begin{equation}
 \label{FtT}
 F_{ \bm{\lambda}, \mathbf{r}} (t, X):=
  \Phi_{X,  \bm{\lambda}, \mathbf{r}}(e^t)
  =\sum_{0<\lambda_n \leq X} \cos(\lambda_n t+\beta_{n})  |r_{n}|, 
\end{equation} 
where $\beta_{n}=\arg r_{n}$.  
 Note that is suffices to show for  $X$ large that   there exists $t \in [1,X]$ such that
$$ F_{ \bm{\lambda}, \mathbf{r}} (t,e^{2X}) = \sum_{0 < \lambda_n \le e^{2X}} \cos(\lambda_n t + \beta_{n}) |r_{n}|  \ge (\alpha_{-}-\e) (\log \log X)^2.$$
This is since the variable change $e^X \to X$, $t=\log x$ establishes  Theorem \ref{generalsumthm}. 
 The problem is reduced to finding the maximal and minimal values of $ F_{ \bm{\lambda}, \mathbf{r}} (t,e^{2X})$.  
We shall focus on determining its maximal value. 

\begin{enumerate}
\item[(i)]  {\bf Connection to the random model.} \\
 The  Effective Linear Independence Conjecture implies
\begin{equation}
  \label{identity}
  \frac{1}{X} \int_{1}^{X} \exp(s F_{ \bm{\lambda}, \mathbf{r}} (t, T) ) \, dt  \sim  
  \mathbb{E} \Big( \exp \Big( s \sum_{0 < \lambda_n \le T} |r_{n}| \cos(\theta_{n}) \Big) \Big)
\end{equation}
where the $\theta_n$ are I.I.D. random variables, uniformly distributed on $[-\pi,\pi]$,  $T= (\log X)^{1-\e}$ with an explicit error term.  The right hand side is the random moment generating function.   This will be established in  Proposition \ref{laplaceprop} 
and is a consequence of the key lemma, 
 Lemma 
\ref{lemmacosineint}. 

\item[(ii)]  {\bf The random moment generating function is large.} \\
By using probability ideas, we show that the right hand side of \eqref{identity} is large.  
Namely, 
\begin{equation}
   \label{rmgflb}
    \mathbb{E} \Big( \exp \Big( s \sum_{0 < \lambda_n \le T} |r_{n}| \cos(\theta_{n}) \Big) \Big) 
    \gg \exp ( (\alpha_{-}(1-o(1))T \log T )
\end{equation}
where $\alpha_{-}$ is the constant given in Assumption 1 \eqref{ass1}.  This lower bound is obtained by 
using independence of random variables, lower bounds for $I_0$ Bessel functions (Lemma \ref{I0lblemma}), 
and the lower bound for $\sum_{0 < \lambda_n \le T} |r_{n}|$ from Assumption 1 in  \eqref{ass1}. 
\item[(iii)] {\bf Large values of  $ F_{ \bm{\lambda}, \mathbf{r}} (t, T):= \Phi_{T}(e^t)$  with $T = (\log T)^{1-\e}$}.  
 In Proposition \ref{measureprop}: We deduce from \eqref{identity} 
and \eqref{rmgflb}
that 
$$F_{ \bm{\lambda}, \mathbf{r}} (t, T)  \ge (\alpha_{-}  -\e') (\log T)^A$$
for  many values of $t \in [1,X]$ for $T = (\log X)^{1-\e}$. 
\item[(iv)]  {\bf Smoothing}.  In Lemma \ref{LemFejer}, a smoothing with Fejer's Kernel is used 
 to relate $F_{ \bm{\lambda}, \mathbf{r}} (t, (\log X)^{1-\e})$ to an average of $F_{ \bm{\lambda}, \mathbf{r}} (t+u, X)$  where $|u| \le (\log X)^A$. 
\item[(v)]  {\bf Mean square estimate}. In Lemma \ref{LemSecondMomentError} it is demonstrated that $F_{ \bm{\lambda}, \mathbf{r}} (t+u,e^{2X})-F_{ \bm{\lambda}, \mathbf{r}} (t+u,X)$ is small on average for $t \in [1,X]$.  This allows
one to obtain large values of $F_{ \bm{\lambda}, \mathbf{r}} (t+u,e^{2X})$ as desired. 
\end{enumerate}
The final element of Theorem \ref{generalsumthm}  is the establishment of Proposition  \ref{MengProp}. 
\begin{itemize}
\item[(vi)] In Proposition  \ref{MengProp}  we establish the technical bound \eqref{rnrmbound} using an idea of Meng \cite{Me}.
\end{itemize}
 Note that 
Proposition \ref{MengProp} is used in Lemma \ref{LemSecondMomentError}. In addition,   Proposition \ref{MengProp}
is also the key new bound which is used in the proof of the independent result Theorem \ref{ANScor}.

Below we have a diagram of the logical implications in the proof of Theorem \ref{generalsumthm}.
\begin{equation}
 \label{logic1}
 \text{Steps (i), (ii), (iii):}  \, \,
 \mybox{Lemma \ref{lemmacosineint}} \implies \mybox{Proposition \ref{laplaceprop}}
 \implies \mybox{Proposition \ref{measureprop} } 
\end{equation}

\begin{equation}
 \label{logic2}
  \text{Steps (iv), (v):}  \, \,
 \mybox{Lemma \ref{LemFejer}} +   \mybox{Lemma \ref{LemSecondMomentError}}
 +   \mybox{Proposition \ref{measureprop}}
 \implies \mybox{Theorem \ref{generalsumthm} } 
\end{equation}

\begin{equation}
\label{logic3}
     \text{Steps (vi):}  \, \,   \mybox{Proposition \ref{MengProp}}  \implies
    \mybox{Lemma \ref{LemSecondMomentError}}.
\end{equation}
The rest of the article is organized as follows.  
In Section \ref{examples}, several further examples are given. 
In Section \ref{keypropsec}, Proposition \ref{measureprop} will be established along 
with  Lemma \ref{lemmacosineint} and Proposition \ref{laplaceprop}.
In Section \ref{fejersec},   Lemma \ref{LemFejer} and   Lemma \ref{LemSecondMomentError} are established. 
In Section \ref{mainthmsec}  Theorem \ref{generalsumthm} is deduced from 
Lemma \ref{LemFejer}, Lemma \ref{LemSecondMomentError}, and Proposition \ref{measureprop}.
Finally, in Section \ref{Mengsec}, Proposition \ref{MengProp} is established. 


\section{Further examples} \label{examples}
In this section, we provide several examples of error terms and their predicted size. 
\subsection{Prime number theorem for automorphic $L$-functions}
Let $d \in \mathbb{N}$ and let $\piup$ be an irreducible unitary cuspidal automorphic representation of ${\rm GL}_d(\mathbb{A}_\mathbb{Q})$. Let $L(s, \piup)$ be the automorphic $L$-function attached to $\piup$ and 
its associated prime counting function is 
$\psi(x, \piup)=\sum_{n \le x} \Lambda(n) a_{\piup}(n)$.  Properties of such functions may be found  in \cite{LY}.
There exists a constant $0 \le \theta < \frac{1}{2}$ (see \cite[Proposition 4.2]{ANS} such that, for all $x>1$ and $T\geq 2$ we have 
\begin{equation}\label{802}
\psi(x, \piup)-\delta(x, \piup)=R_\piup-\sum_{|\Im(\rho_{\pi})|\leq T}\dfrac{x^{\rho_{\pi}}}{\rho_{\pi}}+
O\Big(\dfrac{x^{\theta+1}\log^2 x}{T}+x^\theta \log{x}+\dfrac{x\log^2{T}}{T\log x}
+\dfrac{x\log{T}}{T}\Big),
\end{equation}
where 
$$\delta(x, \piup)= \left\{  \begin{array}{ll} \frac{x^{1+i\tau_0}}{1+i\tau_0}&{\rm if}~L(s, \piup)=\zeta(s-i\tau_0),\\ 0&{\rm otherwise},\end{array} \right.$$
 $\rho_{\pi}$ runs over the nontrivial zeros of $L(s, \piup)$ and $R_\piup \in \mathbb{C}$.
Choosing $T=X^2$, it  follows from the identity $\Re(\frac{x^{-i \gamma}}{\frac{1}{2} -i \gamma} )
= \Re(\frac{x^{i \gamma}}{\frac{1}{2} +i \gamma} )$ that 
\begin{align*}
  \frac{\Re \psi(x, \piup)- \Re \delta(x,\pi)}{\sqrt{x}}
  = -2 
     \Re \Bigg(
  \sum_{\substack{ 0 < \gamma \le X^2 \\
  \gamma \in S_{\pi}}} \frac{1}{2} \cdot \frac{x^{i \gamma}}{\frac{1}{2}+i \gamma}
   \Bigg) + O(1), 
\end{align*}
for $2 \le x \le X$ where $S_{\pi}= \{ \gamma_{\pi} \ | \ \gamma_{\pi} >0 \}  \cup  \{ -\gamma_{\pi} \ | \ \gamma_{\pi} <0 \}  $.
Since $N_{\pi}(T) \sim \frac{dT}{2 \pi} \log(T)$ as $T \to \infty$, it follows that
\begin{align*}
 \sum_{\substack{0 < \gamma < T \\ \gamma \in S}}
  \frac{1}{| \frac{1}{2}+i \gamma|}
  = \sum_{0 < |\gamma_{\pi}| < T} \frac{1}{|\rho_{\pi}|}
  \sim \sum_{1 \le |\gamma_{\pi}| \le T} \frac{1}{|\gamma_{\pi}|}
  \sim \int_{1}^{T} \frac{1}{t} dN_{\pi}(t) 
  \sim \frac{d}{\pi} \int_{1}^{T} \frac{(\log t)}{t} \, dt 
  \sim \frac{d}{2 \pi} (\log T)^2.
\end{align*}
We label the countable set $S_{\pi}$ as $S_{\pi} = (\lambda_n)_{n \in \mathbb{N}}$
where the sequence is non-decreasing. 
By an application of Theorem \ref{conjthm} we obtain the following. 
\begin{thm}
Assume the Generalized Effective Linear Independence Conjecture (Conjecture \ref{GELI})
for the sequence $S_{\pi} = (\lambda_n)_{n  \in \mathbb{N}}$.  Then we have 
\begin{equation}
 \limsup_{x \to \infty}  \frac{\Re \psi(x, \piup)- \Re \delta(x,\pi)}{\sqrt{x} (\log \log \log x)^2}
 \ge \frac{d}{2 \pi}
 \text{ and }
  \liminf_{x \to \infty}  \frac{\Re \psi(x, \piup)-\Re \delta(x,\pi)}{\sqrt{x} (\log \log \log x)^2}
 \le  -\frac{d}{2 \pi}.
\end{equation}
\end{thm}
This leads to the following conjecture. 
\begin{conj}
For $\piup$ an irreducible unitary cuspidal automorphic representation of ${\rm GL}_d(\mathbb{A}_\mathbb{Q})$, we have 
\begin{equation}
 \limsup_{x \to \infty}  \frac{\Re \psi(x, \piup)- \Re \delta(x,\pi)}{\sqrt{x} (\log \log \log x)^2}
 = \frac{d}{2 \pi}
 \text{ and }
  \liminf_{x \to \infty}  \frac{\Re \psi(x, \piup)- \Re \delta(x,\pi)}{\sqrt{x} (\log \log \log x)^2}
 = -\frac{d}{2 \pi}.
\end{equation}
\end{conj}
This gives an example of an application of a linear independence conjecture involving 
sequences other than the ordinates of the zeros of zeta. 

\subsection{Primes in arithmetic progressions}
Let $a,q$ be relatively prime positive integers and $\pi(x;q,a) = \#\{p\le x\colon p\text{ prime, }p\equiv a\mod q\}$. Define the normalized error term
\[
\widetilde{E}(x;q,a) = \frac{\log x}{\sqrt x} \Big( \pi(x;q,a)-\frac{\pi(x)}{\phi(q)} \Big).
\]
Then by  \cite[Lemma 2.1]{RS} we have 
\begin{equation}
\widetilde{E}(x;q,a) = -\frac{c(q,a)}{\phi(q)} + \frac{1}{\phi(q)} \mostchisum \bar\chi(a) \frac{\psi(x,\chi)}{\sqrt x} + O\bigg( \frac{1}{\phi(q) \log x} \bigg)
\label{Exqa in terms of Exchi}
\end{equation}
where $\chi \mod q$  ranges over Dirichlet characters modulo $q$ and $\psi(x,\chi) = \sum_{n \le x} \Lambda(n) \chi(n)$.  Further, we have the explicit formula
\begin{equation}
  \label{psixchiexplicit}
  \frac{\psi(x,\chi)}{\sqrt{x}} = - \sum_{\substack{|\gamma| \le T \\
  L(\frac{1}{2}+i \gamma,\chi)=0}} \frac{x^{i \gamma}}{\rho}
  +O \Big(\frac{x \log^2(xT)}{T} + \log x \Big). 
\end{equation}
Combining these formulae, we find that 
\begin{equation}
   \widetilde{E}(x;q,a) = - \frac{1}{\phi(q)} \mostchisum \bar\chi(a) 
   \Bigg(\sum_{\substack{|\gamma| \le T \\
  L(\frac{1}{2}+i \gamma,\chi)=0}} \frac{x^{i \gamma}}{\rho}
  \Bigg) + O\Big( \frac{x \log^2(xT)}{T} + \log x  \Big). 
\end{equation}
A standard calculation shows that 
\[
   \mostchisum \bar\chi(a) 
   \Bigg(\sum_{\substack{0 < |\gamma| \le T \\
  L(\frac{1}{2}+i \gamma,\chi)=0}} \frac{x^{i \gamma}}{\rho}
  \Bigg) = 2\Re \bigg( \mostchisum \chi(a) \Lsum{0<\gamma<T} \frac{x^{i\gamma}}{\frac12+i\gamma} \bigg). 
\]
The proof of this may be found within the proof of \cite[Proposition 2.3, pp.130-132]{FM}.
The contribution from possible zeros at $\gamma=0$ is $O(\log q)$.  Therefore 
choosing $T=X^2$, we obtain 
\begin{equation}
   \widetilde{E}(x;q,a) = - \frac{1}{\phi(q)} 2 \Re \mostchisum \bar\chi(a) 
   \Bigg(\sum_{\substack{0 < \gamma  \le X^2 \\
  L(\frac{1}{2}+i \gamma,\chi)=0}} \frac{x^{i \gamma}}{\rho}
  \Bigg) + O_q ( 1 )
\end{equation}
for $2 \le x \le X$.  This is now in the form to apply Theorem \ref{conjthm} and
it suffices to evaluate
\begin{equation}
  \mostchisum 
   \sum_{\substack{0 < \gamma  \le T \\
  L(\frac{1}{2}+i \gamma,\chi)=0}} \frac{1}{|\rho|}.
\end{equation} 
Let 
\begin{equation}
 N_q(T) = \mostchisum 
   \sum_{\substack{0 < \gamma  \le T \\
  L(\frac{1}{2}+i \gamma,\chi)=0}} 1.
\end{equation}
By \cite[Lemma 2.4]{La}, 
$N_q(T) = \frac{(\phi(q)-1)T \log T}{2 \pi} + O_q(T)$
and thus partial summation  implies  
\begin{equation}
  \mostchisum 
   \sum_{\substack{0 < \gamma  \le T \\
  L(\frac{1}{2}+i \gamma,\chi)=0}} \frac{1}{|\rho|}
  \sim \frac{(\phi(q)-1)}{2 \pi} (\log T)^2. 
\end{equation} 
Let $S_q =  \cup_{\chi \ne \chi_0} \{\gamma >0 \ | \  L(\frac{1}{2}+i \gamma,\chi)=0 \}$
and label the elements of $S_q$ as $(\lambda_n)_{n=1}^{\infty}$ where the sequence is non-decreasing. Then Theorem \ref{conjthm} establishes the following.
\begin{thm}
Assume the Generalized Effective Linear Independence Conjecture (Conjecture \ref{GELI})
for the sequence $S_{q} = (\lambda_n)_{n  \in \mathbb{N}}$.  Then we have 
\begin{equation}
 \limsup_{x \to \infty}  \widetilde{E}(x;q,a)
 \ge  \frac{(1-\frac{1}{\phi(q)})}{2 \pi} 
 \text{ and }
  \liminf_{x \to \infty}  \widetilde{E}(x;q,a)
 \le  -\frac{(1-\frac{1}{\phi(q)})}{2 \pi}.
\end{equation}
\end{thm}
This leads to the following conjecture. 
\begin{conj}
\begin{equation}
 \limsup_{x \to \infty}  \widetilde{E}(x;q,a)
 =  \frac{(1-\frac{1}{\phi(q)})}{2 \pi} 
 \text{ and }
  \liminf_{x \to \infty}  \widetilde{E}(x;q,a)
 =  -\frac{(1-\frac{1}{\phi(q)})}{2 \pi}.
\end{equation}
\end{conj}

\subsection{Classical prime number sums}
The Chebyshev function $\vartheta(x) = \sum_{p \le x} \log p$ satisfies the identity
$\psi(x) = \sum_{j=1}^{\infty} \vartheta(j^{\frac{1}{n}})$ and thus 
\begin{equation}
  \label{psithetaid}
  \psi(x) 
  = \vartheta(x) + x^{\frac{1}{2}} +o(x^{\frac{1}{2}}). 
\end{equation}
Therefore from identity \eqref{psithetaid} and Conjecture \ref{Montgomeryconjecture} we obtain 
\begin{conj} \label{thetaconj}
\begin{equation}
     \label{limsupinftheta}
   \limsup_{x \to \infty} \frac{\vartheta(x) -x}{\sqrt{x} (\log \log \log x)^2} = \frac{1}{2 \pi} \text{ and }
     \liminf_{x \to \infty} \frac{\vartheta(x) -x}{\sqrt{x} (\log \log \log x)^2} = -\frac{1}{2 \pi}.
\end{equation}
\end{conj}

The Mertens sum  satisfies 
the identity
\begin{equation}
    \sum_{p \le x} \frac{1}{p}=  \log \log x + M 
    + \frac{\vartheta(x)-x}{x \log x} + O \Big( \frac{1}{\log x} \Big)
\end{equation}
where $M=0.26149 72128 4764 \ldots 
$ is the Meissel-Mertens constant (see OEIS: A077761). 
Based on this identity and Conjecture \ref{thetaconj}, we have
\begin{conj}
\begin{equation}
  \label{limsupMertconj}
   \limsup_{x \to \infty} \frac{\sqrt{x} \log x}{(\log \log \log x)^2} \Bigg( \sum_{p \le x} \frac{1}{p} -   \log \log x -M \Bigg) = \frac{1}{2 \pi}
\end{equation}
and 
\begin{equation}
   \label{liminfMertconj}
    \liminf_{x \to \infty}\frac{\sqrt{x} \log x}{(\log \log \log x)^2} \Bigg( \sum_{p \le x} \frac{1}{p} -   \log \log x -M \Bigg)  = -\frac{1}{2 \pi}.
\end{equation}
\end{conj}
Note that Lamzouri's result  \eqref{limsuppsilb}  or an application of Theorem \ref{conjthm} will imply inequalities for the $\limsup$ terms in  \eqref{limsupinftheta},  \eqref{limsupMertconj} on the assumption of Conjecture \ref{ELI}.  Similarly, one can establish inequalities for the $\liminf$ terms.

\section{Proof of the key Proposition  \ref{measureprop}}
\label{keypropsec}

One of the key ideas in the study of limiting distributions of functions, is to show that the function under 
consideration can be modelled by a random sum.  In order to study this distribution, one often computes 
high moments or the Laplace transform of the random sum.  
This is an old idea and has been studied by many researchers including Wintner \cite{Wi2}.  
In recent years, this idea has been effectively employed by Granville and Soundararajan \cite{GS1}, \cite{GS2}
and by Lamzouri \cite{La1}, \cite{La2}, \cite{La3}.  What is notable in these results
is that the authors establish very strong uniform bounds for high moments of the distribution under consideration and then use this to exhibit extreme values of the function of interest. 

This section contains the proofs of Step (i), Step (ii), Step (iii)  as outlined above. 
The key result in the article is the following proposition. 
\begin{pro}\label{measureprop}
Assume the Generalized Effective LI conjecture (Conjecture \eqref{GELI}).
Let $ \mathbf{r}=(r_n)_{n \in \mathbb{N}}$
 and $\bm{\lambda}= (\lambda_n)_{n \in \mathbb{N}} $
  be  sequences of complex numbers satisfying 
 Assumption 1 \eqref{ass1} and Assumption 4 \eqref{ass4}. Let $\varepsilon,\e_1>0$ be sufficiently small and fixed.  There exists a positive constant $c_2$ such that for $X$ large and  all real numbers $2\leq T\leq (\log X)^{1-\varepsilon}$ we have
\begin{equation}\label{EqMeasureLargeF} 
\frac{1}{X}\textup{meas} \{t\in [1, X] : F_{ \bm{\lambda}, \mathbf{r}} (t, T)> (\alpha_{-}-\e_1) (\log T)^A\} \gg e^{-c_2 T\log T} ,
\end{equation}
 where meas denotes the Lebesgue measure on $\mathbb{R}$. Furthermore, we also have   
 \begin{equation}\label{EqMeasureSmallF} \frac{1}{X}\textup{meas} \{t\in [1, X] : F_{ \bm{\lambda}, \mathbf{r}} (t, T)<- (\alpha_{-}-\e_1) (\log T)^A\}\gg e^{-c_2 T\log T}.
 \end{equation}
\end{pro}
This is a direct consequence of the Generalized  Effective LI conjecture, we prove that as $t$ varies in $[1, X]$, the moment generating function of $F_{ \bm{\lambda}, \mathbf{r}} (t, T)$ is close to that of its corresponding probabilistic random model in a large range, if $T\leq (\log X)^{1-\varepsilon}$.   This is a generalization of Proposition 2.1 of \cite{La0} to arbitrary sequences 
$\bm{\lambda}= (\lambda_n)_{n \in \mathbb{N}} $ and $ \mathbf{r}=(r_n)_{n \in \mathbb{N}}$. One key difference with 
\cite[Proposition 2.1]{La0} is that the constant $c_1$ there is given by $(\alpha_{-}-\e_1)$ here where $\alpha_{-}$ is the 
constant from Assumption 1 \eqref{ass1}.   This is  the done by careful bookkeeping which keeps track of the exact size of the constant $c_2$ and uses an improved bound for the modified $I_0$ Bessel function and a careful choice of one of the parameters $U$ in the argument.  

Note that Proposition \ref{laplaceprop} and Lemma \ref{lemmacosineint} below are essentially  Proposition 3.1 and Lemma 3.2 of \cite{La0}.   The proofs will not be given as they are nearly identical. 
The only differences are that the general sequence $\bm{\lambda}= (\lambda_n)_{n \in \mathbb{N}} $ replaces the ordinates of the zeros of zeta $(\gamma_n)_{n \in \mathbb{N}} $ and instead of invoking Conjecture \ref{ELI} (ELI), Conjecture \ref{GELI} (GELI) is invoked.  These are the key results that connect the sums $F_{ \bm{\lambda}, \mathbf{r}} (t, T)$ to their random versions and this is where GELI is crucially used.  

\begin{pro}\label{laplaceprop}
Assume the Generalized Effective LI conjecture (Conjecture \ref{GELI}).
Let $\bm{\lambda}= (\lambda_n)_{n \in \mathbb{N}} $ and $ \mathbf{r}=(r_n)_{n \in \mathbb{N}}$
  be  sequences of complex numbers satisfying 
 Assumption 1 \eqref{ass1}.
  Let $\varepsilon>0$ be small and fixed. Let $X$ be large and $2\leq T\leq (\log X)^{1-\varepsilon}$ be a real number.  Let 
\begin{equation}
  \label{c0}
c_0 = \frac{1}{27 \pi \alpha_{+}}.
\end{equation} 
For all real numbers $s$ with $|s|\leq c_0 T(\log T)^{1-A} $ we have 
\begin{equation}
  \label{integraltorandom}
\frac{1}{X}\int_2^X \exp\big(s F_{ \bm{\lambda}, \mathbf{r}} (t, T) \big)dt= \ex\left(\exp\left(s\sum_{0<\lambda_n \leq T}  |r_{n}|\cos(\theta_{n})\right)\right)+ O\left(e^{-N_{\bm{\lambda}}(T)}\right).
\end{equation}
\end{pro}

Proposition \eqref{laplaceprop} is a consequence of the following lemma. 
This lemma demonstrates  on the  Generalized Effective LI conjecture (Conjecture \ref{GELI}) that the sequence $(\cos(\lambda_n  t+ \beta_{n}))_{n \in \mathbb{N}}$
behaves like  $ ( \cos(\theta_n) )_{n \in \mathbb{N}}$ where the $\theta_n$ are 
I.I.D. random variables.  Observe that GELI allows for the error term of the size 
$O(X^{-2/3})$. 

\begin{lem}\label{lemmacosineint}
Assume the Generalized Effective LI conjecture (Conjecture \eqref{GELI}).
Let $\bm{\lambda}= (\lambda_n)_{n \in \mathbb{N}} $
  be a  sequence of complex numbers.
  Let $\varepsilon>0$ be small and fixed. Let $X$ be large and $2\leq T\leq (\log X)^{1-\varepsilon}$ be a real number. Let $k\leq N_{\bm{\lambda}}(T)$ be a positive integer and $0<\lambda_{n_1}, \dots, \lambda_{n_k}<T$ are elements  (not necessarily distinct) of the sequence $\bm{\lambda}= (\lambda_n)_{n \in \mathbb{N}} $.  For any real numbers $\beta_1, \dots, \beta_k$, we have 
$$ \frac1X\int_{1}^X\cos(\lambda_{n_1} t+\beta_1)\cos(\lambda_{n_2} t+\beta_2)\cdots \cos(\lambda_{n_k} t+\beta_k) dt= \ex\left(\cos(\theta_{n_1}) \cdots\cos(\theta_{n_k})\right) + O\left(X^{-2/3}\right). 
$$
\end{lem}

The moment generating function in \eqref{integraltorandom} is closely related to the modified Bessel function
 of the first kind   defined by  
\begin{equation}
  \label{I0a}
 I_0(t) =  \sum_{n=0}^{\infty} \frac{1}{(k!)^2} \Big( \frac{t}{2} \Big)^{2k}
\end{equation}
 (see \cite[p.375 , eq. 9.6.10]{AS})
 and which also possesses the integral representation  
\begin{equation}
  \label{I0b}
 I_0(t) = \frac{1}{2 \pi} \int_{-\pi}^{\pi} e^{t \cos(u)} \, du
\end{equation}
 (see \cite[p.181, eq. (4)]{Wa}).
We require the following lower bound. 
\begin{lem} \label{I0lblemma}
Let $\e \in (0,\frac{\pi}{2})$. 
For all $t \ge 0$, 
\begin{equation}
 \label{I0lb}
 I_0(t) \ge \max\Big(1, \frac{\e}{2 \pi} e^{t(1-\e^2/2)} \Big).
\end{equation}
\end{lem}
\begin{proof}
Note that for all $t \in \mathbb{R}$, the bound $I_0(t) \ge 1$ follows directly from the series representation \eqref{I0a}. 
Let $\e \in (0,\frac{\pi}{2})$.  Then  for $t \ge 0$,  we have 
\begin{equation}\label{Bessellb}
I_0(t)= \frac{1}{2\pi}\int_{-\pi}^{\pi}e^{t\cos u} du\geq \frac{1}{2\pi} \int_0^{\e} e^{t\cos u} du \geq
\frac{1}{2 \pi}  e^{t \cos(\e)} \int_{0}^{\e} \, du 
= \frac{\e}{2 \pi}  e^{t \cos(\e)}.
\end{equation}
The second lower bound in \eqref{I0lb} follows  from the inequality $\cos(\e) \ge 1-\e^2/2$, which is valid for all $\e \in \mathbb{R}$, hence for all 
$\e \in (0,\frac{\pi}{2})$.  
\end{proof}
Note that Lamzouri has proven the  better bound $I_0(t) \ge \frac{e^t}{10 \pi t}$ in \cite[Lemma 4.2, p. 162]{La}.

We now prove Proposition \ref{laplaceprop}.
\begin{proof}[Proof of Proposition \ref{laplaceprop}]  
In the proof of the corresponding result  in \cite{La0} $c_0$ arises from the inequality:

\begin{align*}
  E_1 & \ll  \sum_{k>R} \frac{(|s|H(T))^k}{k!}   \ll \frac{1}{\sqrt{2 \pi R}}\sum_{k > R}  \Big( \frac{e|s| H(T)}{k} \Big)^k
\end{align*}
where  $E_1$ is the error term in \cite[Proposition 3.1]{La0}, 
$H(T) = \sum_{0 < \lambda_n \le T} r_n$, and $R=N_{ \bm{\lambda}}(T)$, defined in \eqref{Nlambda}.  
Note that  Stirling's 
 formula $k! \sim \sqrt{2 \pi k} (k/e)^k$ has been applied. 
Further
\[
  a  :=  \frac{e|s| H(T)}{k}   \le 
   \frac{e|s| H(T)}{N_{ \bm{\lambda}}(T)} \le \frac{e|s| \alpha_{+}(\log T)^A}{
   \kappa_{-} T (\log T) } \le
   \frac{ \pi e|s| \alpha_{+}(\log T)^{A-1}}{ \kappa_{-}T}.
\]
by the upper bound in \eqref{ass1} and the lower bound
  $N_{ \bm{\lambda}}(T) \ge \kappa_{-} T (\log T)$, for $T$ sufficiently large. 
Now suppose $|s| \le c_0 T (\log T)^{1-A}$ and thus $|a| \le \frac{\pi e \alpha_{+}} {\kappa_{-}} c_0 \le  9 \pi  \alpha_{+} c_0 = 1/3$
since we recall that  $c_0 = \frac{1}{27 \pi \alpha_{+}}$ in \eqref{c0}. 
It follows that
\[
  E_1 \ll \frac{1}{\sqrt{2 \pi R}}\sum_{k > R}  a^k \ll a^{R} = e^{R \log(a)} \le e^{-R \log(3)} < e^{-R}. 
\]
\end{proof}

We end this section by proving Proposition \ref{measureprop}. 

\begin{proof}[Proof of Proposition \ref{measureprop}]
We begin with the proof of  \eqref{EqMeasureLargeF}.
Let $\e >0$ and $s=c_0 T(\log T)^{1-A}$, where $c_0$  is defined in \eqref{c0}.
We begin by establishing a lower bound for the expected value on the right hand side of \eqref{integraltorandom}.
By the independence of the $\theta_n$ it follows that
\begin{equation}
  \label{expectationeven}
 \ex\left(\exp\left(s\sum_{0<\lambda_n \leq T}  |r_{n}|\cos(\theta_{n})\right)\right)= \prod_{0<\lambda_n \leq T}\ex\Big(\exp\big(s |r_{n}|\cos(\theta_{n})\big)\Big)=\prod_{0<\lambda_n \leq T} I_0(s|r_{n}|), 
\end{equation}
where $I_0(t)$ is defined in \eqref{I0a}.  We choose $U=T^{b}$ with $b \in (0,1)$ and we shall let $b$ be arbitrarily close to $1$.  For those $\lambda_n \in (U,T]$, we apply the 
the first bound in \eqref{I0lb} and for those $\lambda_n \in (0, U]$ we apply  the second bound 
in \eqref{I0lb} to obtain
\begin{equation}\label{LowerBoundLaplace}
\ex\left(\exp\left(s\sum_{0<\lambda_n \leq T}  |r_n|\cos(\theta_n)\right)\right) \geq  
  \prod_{0<\lambda_n\leq U} C_{\varepsilon} e^{s|r_n|(1-\varepsilon^2/2)}
\end{equation}
where $C_{\e}= \frac{\e}{2 \pi}$.  We have
\begin{align*}
    \prod_{0<\lambda_n\leq U} C_{\varepsilon} e^{s|r_n|(1-\varepsilon^2/2)}
    = \exp \Big(s(1-\varepsilon^2/2) \sum_{0<\lambda_n\leq U} |r_n| +\log(C_{\varepsilon}) 
    N_{ \bm{\lambda}}(U) 
   \Big)
\end{align*}
and then observe that 
\begin{align*}
  s(1-\varepsilon^2/2) \sum_{0<\lambda_n\leq U} |r_n| +\log(C_{\varepsilon}) N_{ \bm{\lambda}}(U)  
 & \ge  (c_0  T (\log T)^{1-A})  (1-\varepsilon^2/2)  \alpha_{-} (\log U)^{A}  + \log(C_{\varepsilon}) N_{ \bm{\lambda}}(U)   \\
 & = c_0 \alpha_{-} \cdot b^A \cdot (1-\e^2/2) ( T \log T) + O_{\e}(T^b \log T) \\ 
 & \ge s \alpha_{-} (1-\e') ( \log T)^{A}
\end{align*}
where $b$ is taken sufficiently close to 1, $\e'$ is taken sufficiently small,  and $T$ is sufficiently large.
Note that we have applied the upper bound for $N_{\lambda}(U)$ in \eqref{ass4}. Therefore we have 
\begin{equation}
  \ex\left(\exp\left(s\sum_{0<\lambda_n \leq T}  |r_n|\cos(\theta_n)\right)\right)
  \ge  \exp \Big(  s  \alpha_{-}  (1-\e')( \log T)^{A} \Big). 
 \end{equation}
 Combining this last bound with  Proposition  \ref{laplaceprop}, we deduce that 
\begin{equation}\label{LowerBoundLaplaceF}
\frac{1}{X}\int_1^X \exp\big(s F_{ \bm{\lambda}, \mathbf{r}} (t, T) \big)dt \gg 
\exp ( sc_3 ( \log T)^{A} )
\end{equation}
where 
\begin{equation}
 \label{c3}
  c_3 =  \alpha_{-}  (1-\e'). 
\end{equation}
Now let
\begin{equation}
  \label{AplusX}
 \mathcal{A}_{+}(X) = \Big\{ t \in [1,X] \ | \ F_{ \bm{\lambda}, \mathbf{r}} (t, T)\geq (1-\e')  c_3 (\log T)^A \Big\}.
\end{equation}
Then we have 
\begin{align*}
\int_1^X \exp\big(s F_{ \bm{\lambda}, \mathbf{r}} (t, T) \big)dt &= \int_{\mathcal{A}_{+}(X)} \exp\big(s F_{ \bm{\lambda}, \mathbf{r}} (t, T) \big)dt+ \int_{[1, X]\setminus\mathcal{A}_{+}(X)} \exp\big(s F_{ \bm{\lambda}, \mathbf{r}} (t, T) \big)dt\\
&= \int_{\mathcal{A}_{+}(X)}  \exp\big(s F_{ \bm{\lambda}, \mathbf{r}} (t, T) \big)dt +O\left(X\exp\left( (1-\e') c_3 s (\log T)^A\right)\right).
\end{align*}
Hence, it follows from \eqref{LowerBoundLaplaceF} that 
\begin{equation}
 \label{intlb}
\int_{\mathcal{A}_{+}(X)}  \exp\big(s F_{ \bm{\lambda}, \mathbf{r}} (t, T) \big)dt\gg 
X \exp\left(c_3 s(\log T)^A\right).
\end{equation}
By Assumption 1, $
 F_{ \bm{\lambda}, \mathbf{r}} (t, T) \le \alpha_{+} (\log T)^A$
and thus 
\begin{align*}
  \int_{\mathcal{A}_{+}(X)}  \exp\big(s F_{ \bm{\lambda}, \mathbf{r}} (t, T) \big)dt
  \le      \int_{\mathcal{A}_{+}(X)}  \exp\big(s  \alpha_{+} (\log T)^A \big)dt
  \le  \textup{meas}(\mathcal{A}_{+}(X)) \exp\big(s  \alpha_{+} (\log T)^A \big).
\end{align*}
It follows that 
\begin{align*}
  \textup{meas}(\mathcal{A}_{+}(X))  & \ge \exp\big(-s  \alpha_{+} (\log T)^A \big)   \int_{\mathcal{A}(X)}  \exp\big(s F_{ \bm{\lambda}, \mathbf{r}} (t, T) \big)dt \\
  &  \ge \exp\big(-s  \alpha_{+} (\log T)^A \big) \cdot X \exp\left(c_3 s(\log T)^A\right) \\
  &  \ge X  \exp\big(-s  (\alpha_{+} -c_3) (\log T)^A \big) 
\end{align*}
by an application of \eqref{intlb}. Therefore we have 
\begin{equation}
  \label{measlb}
 \frac{1}{X}   \textup{meas}(\mathcal{A}_{+}(X))   
 \ge \exp \big(- c_2  T (\log T) \big) 
 \text{ where } c_2 = c_0 (\alpha_{+} -c_3).
\end{equation}
Note that $c_2 >0$ since $c_3 < \alpha_{-} < \alpha_{+}$. Finally, 
note that by \eqref{c3}
\begin{equation}
  \label{numsineq}
   (1-\e')  c_3 =(1-\e')^2 \alpha_{-} > \alpha_{-} (1-\e_1),
\end{equation}
where $\e_1$ is sufficiently small. Hence combining  \eqref{AplusX}, \eqref{measlb}, and
\eqref{numsineq}, we establish \eqref{EqMeasureLargeF}. 

The proof of  \eqref{EqMeasureSmallF} is very similar.  In this case, we would define $s=-c_0 T(\log T)^{1-A}$. 
Since $I_0(t)$ is an even function, it follows that the right hand side of \eqref{expectationeven} is even. 
Thus we obtain 
$$ \frac{1}{X}\int_2^X \exp\big(sF_{ \bm{\lambda}, \mathbf{r}} (t, T))dt \gg \exp(-c_3s (\log T)^A).$$
We would then consider the set 
\begin{equation}
 \mathcal{A}_{-}(X) = \Big\{ t \in [1,X] \ | \ F_{ \bm{\lambda}, \mathbf{r}} (t, T)\le -(1-\e')  c_3 (\log T)^A \Big\}. 
\end{equation}
Following, exactly the same argument as above, we deduce \eqref{measlb} with $\mathcal{A}_{+}(x)$
replaced by $\mathcal{A}_{-}(x)$ and thus we also have \eqref{EqMeasureSmallF}.
\end{proof}

\section{Smoothing the sum}
\label{fejersec}

This section contains Steps (iv) and (v) of the argument. 
A standard technique in analytic number is to smooth an unsmoothed sum.  The first proofs of the prime number theorem make use of this idea \cite{VP1}, \cite{VP2}.   There are many possible ways to smooth a sum. For instance, see \cite[section 5.1, pp.142-144]{MV} for some standard choices.  In this section we make use of the Fejer kernel given by 
\begin{equation}
  \label{Fejer}
K(u):= \left(\frac{\sin(\pi u)}{\pi u}\right)^2.
\end{equation}
Its Fourier transform is the sawtooth function 
\begin{equation}\label{Fejerft}
\widehat{K}(t) := 
\int_{-\infty}^{\infty} K(u) e^{2\pi i t u} du= \max(0, 1-|t|).
\end{equation}
We shall employ the following properties of $K$. 
\begin{lem} We have for any positive reals $T,Z$
\begin{equation}
  \label{Kprops}
 \int_{-\infty}^{\infty} \frac{T}{2\pi} K\left(\frac{Tu}{2\pi}\right)du =1, 
 \end{equation}
 \begin{equation}
   \label{Kprops2}
  \int_{-Z}^{Z} \frac{T}{2\pi} K\left(\frac{Tu}{2\pi}\right)du =1 + O
  \Big(\frac{1}{ZT} \Big). 
\end{equation}
\end{lem}
\begin{proof}
The equality \eqref{Kprops} is given by the change of variable  $v=  \tfrac{T}{2\pi} u$:
\begin{equation}
   \int_{-\infty}^{\infty} \frac{T}{2\pi} K\left(\frac{Tu}{2\pi}\right)du
= \int_{-\infty}^{\infty}  K\left(v\right)du= \widehat{K}(0)= 1.
\end{equation}
The formula in \eqref{Kprops2} is 
\begin{align*}
   \int_{-Z}^{Z} \frac{T}{2\pi} K\left(\frac{Tu}{2\pi}\right)du
   & = \int_{-\infty}^{\infty} \frac{T}{2\pi} K\left(\frac{Tu}{2\pi}\right)du - \int_{|u| >Z}  \frac{T}{2\pi} K\left(\frac{Tu}{2\pi} 
   \right)du \\
   & = 1+ O \Big( T \int_{|u| \ge Z} (Tu)^{-2}  \Big) \\
   & = 1+ O \Big(\frac{1}{TZ} \Big),  
\end{align*}
by \eqref{Kprops} and the bound $K(u) \ll |u|^{-2}$ for $u \ne 0$. 

\end{proof}

The following key identity can be proven exactly as in \cite[Lemma 2.2]{La0}. 
\begin{lem}\label{LemFejer}
Let $ \mathbf{r}=(r_n)_{n \in \mathbb{N}}$
 and $\bm{\lambda}= (\lambda_n)_{n \in \mathbb{N}} $
  be  sequences of complex numbers satisfying 
 Assumption 1 \eqref{ass1}.
 Let $Y\geq T\geq 2$ be real numbers. For any real number $Z\geq (\log Y)^A$ we have 
\begin{equation}\label{FejerMain}
F_{ \bm{\lambda}, \mathbf{r}} (t, T) = \int_{-Z}^{Z} \frac{T}{2\pi} K\left(\frac{Tu}{2\pi}\right) F_{ \bm{\lambda}, \mathbf{r}}(t+u, Y) du +O\left( \frac{1}{T} \sum_{0<\lambda_n \leq T} \lambda_n |r_n| \right).
\end{equation}
\end{lem}

The following lemma is a direct consequence of Proposition \ref{MengProp}.  The proof follows
\cite{La0}.  The only difference is that we may apply the weaker condition given by Assumption 3 \eqref{ass3}. 
In \cite{La0} the condition was $\theta < 3-\sqrt{3}$.

\begin{lem}\label{LemSecondMomentError}
Let $ \mathbf{r}=(r_n)_{n \in \mathbb{N}}$
 and $\bm{\lambda}= (\lambda_n)_{n \in \mathbb{N}} $
  be  sequences of complex numbers satisfying 
 Assumption 3 \eqref{ass3} and Assumption 5 \eqref{ass5}. 
Let $X, Y\geq 2$ be real numbers and $h,Z$ are positive real numbers.  
There exists a positive constant $\eta$ such that 
 \begin{equation}
   \label{doubleintbd}
 \frac{1}{X}\int_1^{X} \left|\int_{-Z}^Z h K(hu)
 \Bigg( \sum_{Y <\lambda_n \leq e^{2X}} e^{i \lambda_n (t+u)} r_n \Bigg) 
  du\right|^2dt \ll Y^{-\eta}.
  \end{equation}
 \end{lem}
 
 \begin{proof}
 Expanding out the square we find that the left hand side of \eqref{doubleintbd}
 is 
 \begin{equation}
   \label{lhsexpansion}
\int_{-Z}^Z\int_{-Z}^Z h^2 K(hu) K(hv) \left(\frac{1}{X}\int_2^X
\Bigg( \sum_{Y <\lambda_n \leq e^{2X}} e^{i \lambda_n (t+u)} r_n \Bigg)
\overline{\Bigg( \sum_{Y <\lambda_n \leq e^{2X}} e^{i \lambda_n (t+u)} r_n \Bigg)}
 dt \right) du dv.
\end{equation}
The inner integral is 
\begin{align*} 
& \frac{1}{X}\int_1^X \sum_{Y<\lambda_m \leq e^{2X}} e^{i\lambda_m (t+u)}r_m \sum_{Y<\lambda_n \leq e^{2X}} e^{-i\lambda_n (t+v)} \overline{r_n}  dt
& = \sum_{Y<\lambda_m, \lambda_n \leq e^{2X}}e^{i(\lambda_m u-\lambda_n v)} r_{m} \overline{r_{n}} \frac{1}{X}\int_1^X e^{i(\lambda_m-\lambda_n) t}dt\\
& \ll \sum_{Y<\lambda_m, \lambda_n \leq e^{2X}} |r_m r_n|\min\left(1, \frac{1}{|\lambda_m-\lambda_n|}\right) \ll  Y^{-\eta}
\end{align*}
where $\eta=2-\theta-2\epsilon$
by an application of Lemma \ref{MengProp}.  Plugging this back in \eqref{lhsexpansion} we find that  
\begin{align*}
 \frac{1}{X}\int_1^{X} \left|\int_{-Z}^Z h K(hu)
 \Bigg( \sum_{Y <\lambda_n \leq e^{2X}} e^{i \lambda_n (t+u)} r_n \Bigg)
  du\right|^2dt 
  & \ll Y^{-\eta} \int_{-Z}^Z\int_{-Z}^Z h^2 K(hu) K(hv) du dv  \\
  & = Y^{-\eta} \Big(  \int_{-Z}^Z  h K(hu)  du \Big)^2 \\
  & \le Y^{-\eta}.  
\end{align*}
 
 \end{proof}

\section{Proof of Theorem \ref{generalsumthm} and Corollaries \ref{Mofxcorr}, \ref{Lofxcorr}}
\label{mainthmsec}

As indicated in the logical diagram \eqref{logic2}, 
Theorem \ref{generalsumthm} will now be deduced from Proposition \ref{measureprop}, Lemma \ref{LemFejer}, and Lemma \ref{LemSecondMomentError}.
 \begin{proof}[Proof of Theorem \ref{generalsumthm}]
 In this argument $\delta, \delta' ,\delta'', \e, \e_2, \e',\e''$ are sufficiently small positive numbers. 
  We set  $Z=(\log X)^A$ and define 
 \begin{equation}
 \mathcal{E}(X) = \Big\{
 t \in  [\sqrt{X}, X-\sqrt{X}] \ | \ \left|\int_{-Z}^Z \frac{T}{2\pi} K\left(\frac{Tu}{2\pi}\right) 
\big(F_{ \bm{\lambda}, \mathbf{r}}(t+u, e^{2X})-F_{ \bm{\lambda}, \mathbf{r}}(t+u, X)\big) du\right|>1
 \Big\}.
 \end{equation}
 By an application of  Lemma \ref{LemSecondMomentError}, we find that 
 \begin{equation}\label{MeasureEpsilon}
\begin{aligned}\text{meas}(\mathcal{E}(X)) &\leq \int_1^X \left|\int_{-Z}^Z \frac{T}{2\pi} K\left(\frac{Tu}{2\pi}\right) \big(F_{ \bm{\lambda}, \mathbf{r}}(t+u, e^{2X})-F_{ \bm{\lambda}, \mathbf{r}}(t+u, X)\big) du\right|^2dt\\
& \leq \int_1^X \left|\int_{-Z}^Z \frac{T}{2\pi} K\left(\frac{Tu}{2\pi}\right) 
\Bigg( \sum_{X <\lambda_n \leq e^{2X}} e^{i \lambda_n (t+u)} r_n \Bigg)
du\right|^2dt \\
&  \ll X^{1-\eta},
\end{aligned}
\end{equation}
where $\eta$ is the constant from Lemma \ref{LemSecondMomentError}.
Let  $T=(\log X)^{1-\delta}$ where $\delta$ is a sufficiently small positive number. 
We define the following set on which $F_{ \bm{\lambda}, \mathbf{r}} (t, T)$ exhibits 
large positive values
\begin{equation}
  \mathcal{A}_{+}(X) = \Big\{ t\in  [\sqrt{X}, X-\sqrt{X}] \ | \ F_{ \bm{\lambda}, \mathbf{r}} (t, T)\geq c_1 (\log T)^A \Big\}, 
\end{equation} 
where $c_1 =(1-\e_2)  \alpha_{-} $ is the constant appearing in Proposition
\ref{measureprop}.
Observe that 
\begin{equation}
\begin{split}
  \label{measdiff}
  & \text{meas}\big(\mathcal{A}_{+}(X)\setminus \mathcal{E}(X)\big)  \\
  &  \ge
   \text{meas}\big( \Big\{ t\in  [\sqrt{X}, X-\sqrt{X}] \ | \ F_{ \bm{\lambda}, \mathbf{r}} (t, T)\geq c_1 (\log T)^A \Big\} \big) -2 \sqrt{X} -
     \text{meas}\big( \mathcal{E}(X)\big)  \\
     & \ge \exp(-c_2 T \log T) - 2 \sqrt{X} - c_5 X^{1-\eta}
\end{split}
\end{equation} 
by an application of Proposition \ref{measureprop} and by the bound  \eqref{MeasureEpsilon}
where $c_5 >0$.  
Here $c_2=c_0 (\alpha_{+} -c_3)$ is the constant from  Proposition \ref{measureprop}.
Note that 
\begin{equation}
  \label{expid}
   \exp(-c_2 T \log T) 
   =  \exp(-c_2 (\log X)^{1-\delta} (1-\delta) \log \log  X). 
\end{equation}
Combining \eqref{measdiff} and \eqref{expid} we find that 
\[
   \text{meas}\big(\mathcal{A}_{+}(X)\setminus \mathcal{E}(X)\big)  \ge X \exp \Big( -\frac{\log X}{\log \log X} \Big),
\]
for $X$ sufficiently large. 
By Lemma  \ref{LemFejer} and Assumption 2 \eqref{ass2}, we have 
\begin{equation}
  \label{Fintid}
 F_{ \bm{\lambda}, \mathbf{r}} (t, T) = \int_{-Z}^{Z} \frac{T}{2\pi} K\left(\frac{Tu}{2\pi}\right)  F_{ \bm{\lambda}, \mathbf{r}}(t+u, X) du + o((\log T)^A). 
\end{equation}
Let $t \in \mathcal{A}_{+}(X)\setminus \mathcal{E}(X)$. 
From the definition of the set $\mathcal{A}_{+}(X)$ it follows that
 \begin{equation}\label{LBIntFejerF}
\int_{-Z}^{Z} \frac{T}{2\pi} K\left(\frac{Tu}{2\pi}\right) F_{ \bm{\lambda}, \mathbf{r}}(t+u, X) du\geq  (1-\delta')c_1(\log T)^A
\end{equation}
for $\delta'>0$. 
Since $t\notin \mathcal{E}(X)$,
it follows from the definition of $\mathcal{E}(X)$ and \eqref{LBIntFejerF} that 
\begin{equation}\label{LBIntFejerF2}
\begin{aligned}
\int_{-Z}^{Z} \frac{T}{2\pi} K\left(\frac{Tu}{2\pi}\right) F_{ \bm{\lambda}, \mathbf{r}}(t+u, e^{2X}) du 
& \geq \int_{-Z}^{Z} \frac{T}{2\pi} K\left(\frac{Tu}{2\pi}\right) F_{ \bm{\lambda}, \mathbf{r}}(t+u, X) du -1\\
& \geq (1-\delta')c_1 (\log T)^A-1.
\end{aligned}
\end{equation}
Next, we use the non-negativity of the Fejer kernel to bound the integral on the left hand side of \eqref{LBIntFejerF}.
We have 
\begin{equation}\label{BoundFejerIntegral}
\begin{aligned}
\int_{-Z}^{Z} \frac{T}{2\pi} K\left(\frac{Tu}{2\pi}\right) F_{ \bm{\lambda}, \mathbf{r}}(t+u, e^{2X}) du
& \leq \max_{|u|\leq Z} F_{ \bm{\lambda}, \mathbf{r}}(t+u, e^{2X})\int_{-Z}^{Z} \frac{T}{2\pi} K\left(\frac{Tu}{2\pi}\right)du  \\
& \leq \max_{|u|\leq Z} F_{ \bm{\lambda}, \mathbf{r}}(t+u, e^{2X})
\end{aligned}
\end{equation}
 since   $\int_{-Z}^{Z} \frac{T}{2\pi} K\left(\frac{Tu}{2\pi}\right)du
\leq  \int_{-\infty}^{\infty} \frac{T}{2\pi} K\left(\frac{Tu}{2\pi}\right)du =1$ by  \eqref{Kprops}. 
Combining  \eqref{LBIntFejerF2} and \eqref{BoundFejerIntegral} yields
\begin{equation}
    \max_{|u|\leq Z} F_{ \bm{\lambda}, \mathbf{r}}(t+u, e^{2X})
  \ge (1-\delta')c_1 (\log T)^A-1.
\end{equation} 
Letting $y=t+u$ where $t \in [\sqrt{X}, X-\sqrt{X}]$ and $|u| \le Z=(\log X)^A$, it follows that $y \in [1,X]$ for $X$ sufficiently large
and thus 
\begin{equation}
\begin{split}
  F_{ \bm{\lambda}, \mathbf{r}}(y,e^{2X})  &  \ge (1-\delta')c_1 (\log T)^A-1 \ge (1-\delta'/2) c_1 (\log T)^{A} \\
  & =  (1-\delta'/2) c_1 (1-\delta)^A \log \log X.
\end{split}
\end{equation}
Therefore, we can take 
\begin{equation}
 C_1 = (1-0.5 \delta') c_1 = (1-0.5 \delta')(1-\e'')(1-\delta)^A \alpha_{-}. 
\end{equation}
By the variable changes $X'=e^X$, $x=e^y$ and since $\delta', \e'',\delta$ are arbitrary positive numbers, we deduce 
that $\Phi_{(X')^2}(x) =F_{ \bm{\lambda}, \mathbf{r}}(y,(X')^2) > \alpha_{-}(1-\e) (\log \log \log X')^A$  
where $x \in [1,X']$ and we establish \eqref{MaxOmega} as desired. 

The proof of \eqref{MinOmega} is very similar to that of \eqref{MaxOmega}.  
Instead of $\mathcal{A}_{+}(X)$, we consider the set
\begin{equation}
 \mathcal{A}_{-}(X) = 
 \Big\{ t\in  [\sqrt{X}, X-\sqrt{X}] \ | \ F_{ \bm{\lambda}, \mathbf{r}} (t, T)\leq -c_1 (\log T)^A \Big\}
\end{equation}
where $F_{ \bm{\lambda}, \mathbf{r}} (t, T)$ takes on negative values. 
Note that
\begin{equation}
\begin{split}
  \label{minineq}
   \int_{-Z}^{Z} \frac{T}{2\pi} K\left(\frac{Tu}{2\pi}\right) F(t+u, e^{2X}) du 
   &  \ge \min_{|u|\leq Z} F(t+u, e^{2X})\int_{-Z}^{Z} \frac{T}{2\pi} K\left(\frac{Tu}{2\pi}\right)du \\
   & = 
    \min_{|u|\leq Z} F(t+u, e^{2X}) \Big( 1+ O \Big( \frac{1}{(\log X)^{A+1-\delta}} \Big) \Big),
\end{split}
\end{equation}
by the second equality in \eqref{Kprops}.  
As $t \notin \mathcal{E}(x)$, we have 
\[
     \int_{-Z}^{Z} \frac{T}{2\pi} K\left(\frac{Tu}{2\pi}\right) F(t+u, e^{2X}) du 
     \le \int_{-Z}^{Z} \frac{T}{2\pi} K\left(\frac{Tu}{2\pi}\right) F_{ \bm{\lambda}, \mathbf{r}}(t+u, X) du +1.
\]
By \eqref{Fintid} and the assumption $t \in \mathcal{A}_{-}(X)$, it follows 
that 
\begin{align*}
 \int_{-Z}^{Z} \frac{T}{2\pi} K\left(\frac{Tu}{2\pi}\right) F_{ \bm{\lambda}, \mathbf{r}}(t+u, X) du
 \le -c_1 (\log T)^{A} +o((\log T)^{A}). 
\end{align*}
Combining the last two inequalities, we obtain  for $t \in  \mathcal{A}_{-}(X)\setminus \mathcal{E}(X) $
\begin{equation}
    \int_{-Z}^{Z} \frac{T}{2\pi} K\left(\frac{Tu}{2\pi}\right) F(t+u, e^{2X}) du 
    \le - (c_1-\delta') (\log T)^{A}.
\end{equation}
Combining the last inequality with \eqref{minineq},  we have for $t \in  \mathcal{A}_{-}(X)\setminus \mathcal{E}(X) $
\begin{equation}
     \min_{|u|\leq Z} F(t+u, e^{2X}) \le  - (c_1-\delta'') (\log T)^{A}.
\end{equation}

 \end{proof}

We end this section by deducing Corollary \ref{Mofxcorr} for the summatory function of the M\"obius function. 
\begin{proof}[Proof of Corollary \ref{Mofxcorr}]
Note that if RH is false, then $M(x) = \Omega_{\pm} (x^{\vartheta})$ for some $\vartheta >\frac{1}{2}$
and if the zeta function has a multiple zero then $M(x) = \Omega_{\pm}(x^{\frac{1}{2}} (\log x))$.  Note that this provides 
a significantly stronger bound than \eqref{LimsupMofx}.

Therefore we may assume 
the truth of the Riemann hypothesis and the simplicity of zeros of $\zeta(s)$. 
In this case, we shall apply Theorem \ref{generalsumthm}. It is well-known that  $M(x)$ has an explicit formula
(see \cite[Theorem 14.27, pp.]{Ti}, \cite[Lemma 4, p. 370]{Ng}).  We apply a version from \cite[eq. (4.18)]{ANS}. 
The assumption $J_{-1}(T) \ll T^{\theta}$ with $\theta <2$ implies that 
 \begin{equation}
  \label{Mxexplicit2}
  \frac{M(x)}{\sqrt{x}} = 2 \Re \Big( \sum_{0 < \gamma < X^2} \frac{x^{i\gamma}}{\rho \zeta'(\rho)} \Big) + O(1)
 \end{equation}
 for all $2 \le x \le X$.  It follows that 
 \begin{equation}
     \label{Mxexplicit2}
   \frac{M(x)}{\sqrt{x}} = 2 \Phi_{X^2,  \bm{\lambda}, \mathbf{r}}(x)+ O(1) 
 \end{equation}
 for $2 \le x \le X$, 
 where $\bm{\lambda}= (\lambda_n)_{n \in \mathbb{N}} = (\gamma_n)_{n \in \mathbb{N}} $ and  $\mathbf{r}=(r_n)_{n \in \mathbb{N}}=((\rho_n \zeta'(\rho_n))^{-1})_{n \in \mathbb{N}}$. 
 We shall apply  Theorem \ref{generalsumthm}  and thus we will now show that these specific sequences  
 $\bm{\lambda}$, $\mathbf{r}$,  satisfy Assumptions 1-3, \eqref{ass1}, \eqref{ass2}, \eqref{ass3}. 
 Note that  \eqref{Jminus1over2T} is the statement 
 \begin{equation}
  \label{J12Tbd}
 a_{-} T (\log T)^{\frac{1}{4}} \le J_{-1/2}(T) \le a_{+} T (\log T)^{\frac{1}{4}}.
\end{equation}
Observe that we have 
 \begin{equation}
  \label{eq1}
  \sum_{0 < \gamma_n \le T} \frac{1}{|\gamma_n \zeta'(\rho_n)|}  \le 
  \sum_{0 < \lambda_n \le T} |r_n| = \sum_{0 < \gamma_n \le T} \frac{1}{|\rho_n \zeta'(\rho_n)|} \sim \sum_{0 < \gamma_n \le T} \frac{1}{|\gamma_n \zeta'(\rho_n)|}, 
 \end{equation}
 assuming $J_{-1/2}(T) \ll T (\log T)^{\frac{1}{4}}$ (see \cite[Lemma 7.2.5 (c), pp.117-119]{NgPhd} for details
 on the step showing $\sim$ in the above equation). By partial summation, 
 \begin{equation}
   \label{eq2}
     \sum_{0 < \gamma_n \le T} \frac{1}{|\gamma_n \zeta'(\rho_n)|} 
     = \frac{J_{-1/2}(T)}{T} + \int_{1}^{T}  \frac{J_{-1/2}(t)}{t^2} \, dt. 
 \end{equation}
 Applying the bounds in \eqref{J12Tbd} it follows that 
 \begin{equation}
    \label{eq3}
     a_{-}  \int_{1}^{T} \frac{(\log t)^{\frac{1}{4}}}{t} \, dt  \le  \sum_{0 < \gamma_n \le T} \frac{1}{|\gamma_n \zeta'(\rho_n)|}  \le  a_{+} (\log T)^{\frac{1}{4}}
         + a_{+} \int_{1}^{T} \frac{(\log t)^{\frac{1}{4}}}{t} \, dt
 \end{equation}
 Integrating by parts, we see that $\int_{1}^{T} \frac{(\log t)^{\frac{1}{4}}}{t} \, dt\sim \frac{4}{5} (\log T)^{\frac{5}{4}}$. Combining this with \eqref{eq1}, \eqref{eq2}, and \eqref{eq3} we find that for any $\varepsilon >0$, there exists $T_0$ such that for $T \ge T_0$
 \begin{equation}
     \tfrac{4a_{-}}{5} (\log T)^{\frac{5}{4}} \le   \sum_{0 < \lambda_n \le T} |r_n|  \le   (\tfrac{4a_{+}}{5}+\varepsilon) (\log T)^{\frac{5}{4}}
 \end{equation}
 and thus \eqref{ass1} holds with $ \alpha_{-}=\tfrac{4a_{-}}{5}$, $ \alpha_{+}=\tfrac{4a_{+}}{5}+\varepsilon$, and $A= \frac{5}{4}$. 
 By \eqref{J12Tbd}, it follows that 
 \[
   \sum_{0 < \lambda_n \le  T} \lambda_n |r_n| \ll   \sum_{0<\gamma_n\leq T} \frac{\gamma_n}{|\rho_n\zeta'(\rho_n)|}\leq J_{-1/2}(T) \ll T (\log T)^{1/4}=o(T  (\log T)^{5/4})
 \]
 and thus \eqref{ass2} holds with $A=\frac{5}{4}$.  Next
 \[
   \sum_{0<\lambda_n\leq T} \lambda_n^2 |r_{n}|^2
   = \sum_{0<\gamma_n\leq T} \gamma_n^2
   \frac{1}{|\rho_n\zeta'(\rho_n)|^2} \le J_{-1}(T) \ll T^{2-\vartheta}, 
 \]
 by the bound  \eqref{Jminus1T} and thus Assumption 3 \eqref{ass3} holds with $\theta=2-\e$. 
 Assumption 4 and 5 follow from the well-known formula $N(T) = \# \{ \gamma_n \le T\} \sim \frac{T}{2 \pi} \log T$. 
 As Assumptions 1-5, \eqref{ass1}, \eqref{ass2}, \eqref{ass3}, \eqref{ass4}, \eqref{ass5} hold, it follows from \eqref{Mxexplicit2} and Theorem
 \ref{generalsumthm}  that 
 \[
   \max_{x\in [2, X]}  \frac{M(x)}{\sqrt{x}} \ge \Big(\frac{2 \cdot 4 a_{-}}{5} -\e \Big) (\log \log \log X)^{\frac{5}{4}}, 
 \]
 by adjusting the value of $\e$. Note that this implies inequality \eqref{LimsupMofx}. The proof of 
 \eqref{LiminfMofx} is similar. 
 This completes the proof. 
\end{proof}

\begin{proof}[Proof of Corollary \ref{Lofxcorr}]
The proof of this is extremely similar to the proof of Corollary \ref{Mofxcorr} and thus we will not provide the details. 
Note that $L(x)$ also has an explicit formula (see \cite{F} and \cite[p. 773, eq. (4.21)]{ANS}). 
In this case $r_n = \frac{\zeta(2 \rho_n)}{\rho_n \zeta'(\rho_n)}$ 
and instead of using the bounds \eqref{Jminus1over2T}, we would use \eqref{Kminus1over2T}.
This amounts to replacing $a_{\pm}$ by $b_{\pm}$ in the preceding argument. 
\end{proof}

\begin{proof}[Proof of Theorem \ref{conjthm}]
The proof of this is also very similar to the proofs of Corollaries \ref{Mofxcorr} and \ref{Lofxcorr}.
By the assumption \eqref{errorbd} we have 
\begin{equation}
 E_{\varphi}(x) = 2 \Phi_{X^2,  \bm{\lambda}, \mathbf{r}}(x)+ O(1) 
\end{equation}
for $2 \le x \le X$. Let $\e \in (0,\alpha)$. Note that we are assuming Assumption 1B \eqref{ass1b} which implies that \eqref{ass1}
holds with $\alpha_{\pm} = \alpha \pm \e$. Further, we are also assuming \eqref{ass2}, \eqref{ass3}, \eqref{ass4}, and \eqref{ass5}.
Therefore, we can apply Theorem \ref{generalsumthm} and thus 
\[
   \max_{x \in [2,X]} E_{\varphi}(x) \ge
    ( 2 \alpha -\e ) (\log \log \log X)^{A}.
\]
This implies the first inequality in 
\eqref{lblimsup}.  The proof of the second inequality is similar. 
\end{proof}

\section{Proof of Proposition \ref{MengProp}}
\label{Mengsec}
This section contains Step (vi) of the argument which is the 
 proof of Proposition \ref{MengProp}.  Recall that we are assuming the 
sequences $(\lambda_n)_{n \in \mathbb{N}}$  and $(r_n)_{n \in \mathbb{N}}$  satisfy Assumption 3 \eqref{ass3}
and  Assumption 5 \eqref{ass5}.
\begin{equation} 
  \label{cond1a}
  \sum_{\lambda_n \le T} \lambda_n^2 |r_n|^2 \ll T^{\theta} \text{ with } \theta <2 
  \text{ and }  \sum_{T < \lambda_n \le T+1} 1 \ll \log (T+2). 
\end{equation}
Note that the second bound implies that $N_{ \bm{\lambda}}(T)  =\sum_{\lambda_n \le T} 1  \ll T \log (T+2) $.
By Cauchy-Schwarz these imply 
\begin{equation}
 \label{firstmoment}
  \sum_{\lambda_n \le T} \lambda_n |r_n|
  \le \Big(   \sum_{\lambda_n \le T} \lambda_n^2 |r_n|^2 \Big)^{\frac{1}{2}} N_{ \bm{\lambda}}(T)^{\frac{1}{2}}
   \ll T^{\frac{\theta+1}{2}} \sqrt{\log (T+2)} 
\end{equation}
and thus by partial summation
\begin{equation}
   \label{weightedfirstmoment}
    \sum_{\lambda_n \le T}  |r_n| \ll T^{\frac{\theta-1}{2}} \sqrt{\log (T+2)}. 
\end{equation} 
With these bounds in hand, we establish the following lemma. 
\begin{lem} \label{pslemma}
Assume that \eqref{cond1a} holds.   \\
(i)
 For $a > \frac{\theta+1}{2}$  and for any sufficiently small $\varepsilon >0$, we have 
\begin{equation}
  \label{tail1}
   \sum_{\lambda_n > T} \frac{|r_n|}{\lambda_{n}^{a-1}}
   \ll \frac{1}{T^{a- ( \frac{\theta+1}{2})-\varepsilon}}. 
\end{equation}
(ii) For $b > \theta$, we have 
\begin{equation}
  \label{tail2}
   \sum_{\lambda_n > T} \frac{|r_n|^2}{\lambda_{n}^{b-2}} \ll \frac{1}{T^{b-\theta}}. 
\end{equation}
\end{lem}
\begin{proof}[Proof of Lemma \ref{pslemma}]
(i) Let $U(x) = \sum_{\lambda_n \le x} \lambda_n |r_n|$ and $\varepsilon>0$.  By partial summation and 
\eqref{firstmoment}, 
\begin{align*}
   \sum_{\lambda_n > T} \frac{|r_n|}{\lambda_{n}^{a-1}}
=  \sum_{\lambda_n > T} \frac{1}{\lambda_{n}^a} (\lambda_n |r_n|) = -\frac{U(T)}{T^{a}} + a \int_{T}^{\infty}
     \frac{U(t)}{t^{a+1}} 
     \ll_{a}  \int_{T}^{\infty}  \frac{dt}{t^{a-( \frac{\theta+1}{2})-\varepsilon}} \ll_{a,\theta, \varepsilon}
   \frac{1}{T^{a- ( \frac{\theta+1}{2})-\varepsilon}}, 
\end{align*}
assuming $\varepsilon < a- (\frac{\theta+1}{2})$.  \\
(ii) Let $V(x) = \sum_{\lambda_n \le x} (\lambda_n |r_n|)^2$. By partial summation and the first inequality 
in \eqref{cond1a}, 
\begin{align*}
    \sum_{\lambda_n > T} \frac{|r_n|^2}{\lambda_{n}^{b-2}} =
    \sum_{\lambda_n > T} \frac{1}{\lambda_{n}^b} (\lambda_n |r_n|)^2=
      -\frac{V(T)}{T^b} + b\int_{T}^{\infty} \frac{V(t)}{t^{b+1}} \, dt \ll_{b} \int_{T}^{\infty} \frac{1}{t^{b-\theta+1}} \, dt \ll_{b,\theta} \frac{1}{T^{b-\theta}}, 
\end{align*}
as desired. 
\end{proof}

With this lemma in hand we now establish the proposition.  Note that this proposition is 
 a formalization of the argument in \cite{Me}.   The main difference in Meng's argument is that
 the sum \eqref{sum2} is broken into many pieces of smaller lengths.  Previously, authors as in \cite{ANS}
 only broke it into four pieces  whereas Meng splits into $1/\epsilon$ pieces where $\epsilon$ is sufficiently small. 
\begin{proof}[Proof of Proposition \ref{MengProp}]
Let $\epsilon$ be a positive number which satisfies
\begin{equation}
 \label{epsilonrange}
0 < \epsilon < \frac{2-\theta}{10}.
\end{equation}
Since $|z|^2=z\bar{z}$, we have
\begin{equation}
\begin{split}
  \label{firstinequality}
\int_{V}^{V+1}\Big|\sum_{T<\lambda_n\leq X}r_ne^{iy\lambda_n}\Big|^2dy
&=\sum_{T<\lambda_n\leq X}\sum_{T<\lambda_m\leq X}r_n\overline{r_m}\int_{V}^{V+1}e^{iy(\lambda_n-\lambda_m)}dy\\
&\ll\sum_{T<\lambda_n\leq X}\sum_{T<\lambda_m\leq X}|r_nr_m|\min\left(1,\dfrac{1}{|\lambda_n-\lambda_m|}\right).
\end{split}
\end{equation}
We shall demonstrate that 
\begin{equation}
  \label{rdoublesumbd}
 \sum_{T<\lambda_n\leq X}\sum_{T<\lambda_m\leq X}|r_nr_m|\min\left(1,\dfrac{1}{|\lambda_n-\lambda_m|}\right)
 \ll   \frac{1}{T^{2-\theta-2\epsilon}}
\end{equation}
and this will establish the proposition since we have the inequality \eqref{firstinequality}.
We decompose the sum in \eqref{rdoublesumbd} as $\Sigma_1+\Sigma_2$
where $\Sigma_1$ is the sum of those terms for which we have $|\lambda_n-\lambda_m|<1$, and $\Sigma_2$ 
is the sum of the terms with $|\lambda_n-\lambda_m| \ge 1$. It suffices to establish
$\Sigma_j \ll \frac{1}{T^{2-\theta-2\epsilon}}$ for $j=1,2$. 

 We begin with the bound for $\Sigma_1$. 
By two applications of Cauchy-Schwarz and  \eqref{ass4} we obtain
\begin{eqnarray}
  \Sigma_1&\ll
  &\sum_{T\leq \lambda_m \leq X}|r_m| 
  \sum_{\lambda_m-1\leq\lambda_n\leq\lambda_m+1} |r_n| \nonumber\\
  &\ll&\left(\sum_{T\leq \lambda_m \leq X}|r_m|^2\right)^{\frac{1}{2}}\left(\sum_{T\leq \lambda_m\leq X}\left(\sum_{\lambda_m-1\leq\lambda_n\leq\lambda_m+1}|r_n|\right)^2\right)^{\frac{1}{2}}\nonumber\\
  &\ll&\left(\sum_{T\leq \lambda_m\leq X}|r_m|^2\right)^{\frac{1}{2}}\left(\sum_{T\leq \lambda_m\leq X}\left(\sum_{\lambda_m-1\leq\lambda_n\leq\lambda_m+1}|r_n|^2\right)\cdot\log\lambda_m\right)^{\frac{1}{2}}. 
\end{eqnarray}
By an application of Lemma \ref{pslemma} (ii) with $b=2$ we find
that 
\begin{equation}
   \Sigma_1 \ll \left( \frac{1}{T^{2-\theta}} \right)^{\frac{1}{2}}\left(\sum_{T\leq \lambda_m\leq X}\sum_{\lambda_m-1\leq\lambda_n\leq\lambda_m+1}|r_n|^2 \log\lambda_n\right)^{\frac{1}{2}}\nonumber\\
 \ll \frac{1}{T^{1-\frac{\theta}{2}}}\left(\sum_{T\leq \lambda_n\leq X}m(\lambda_n)|r_n|^2 \cdot\log\lambda_n\right)^{\frac{1}{2}},\nonumber
\end{equation}
where we have swapped summation in the second sum and where
 $$m(\lambda_n)=\#\{\lambda_m: \lambda_m-1\leq \lambda_n\leq \lambda_m+1\}. $$
 By \eqref{ass5} we have $m(\lambda_n)
 \ll \log \lambda_n \ll \lambda_{n}^{\varepsilon}$ and thus 
 \begin{equation}
 \begin{split}
    \label{Sigma1bd}
     \Sigma_1 
     \ll  \frac{1}{T^{1-\frac{\theta}{2}}}\left(\sum_{T\leq \lambda_n\leq X}|r_n|^2 \cdot \lambda_{n}^{2\varepsilon} \right)^{\frac{1}{2}}
     \ll   \frac{1}{T^{1-\frac{\theta}{2}}} \Big(  \frac{1}{T^{2-\theta-2\varepsilon}} \Big)^{\frac{1}{2}}
     = \frac{1}{T^{2-\theta-\varepsilon}}.
 \end{split}
 \end{equation}
We write $\Sigma_2$ as follows,
\begin{equation}\label{sum2}
  \Sigma_2=\sum_{T\leq \lambda_m \leq X} r_m \sum_{\substack{T\leq\lambda_n\leq X\\ |\lambda_m-\lambda_n|>1}}\frac{r_n}{|\lambda_m-\lambda_n|}.
\end{equation}
We select $N$ to be a natural number defined to be $N = \lfloor \frac{1}{\epsilon} \rfloor +2$. Note that $N \ge 3$ and 
 that 
\begin{equation}
\label{Ninequality}
\frac{1}{\epsilon} \le N \le \frac{3}{\epsilon}. 
\end{equation}
Next, we define the sequence of constants 
\begin{equation}
 \label{ajs}
a_1=1-\frac{1}{N},\quad a_2=1-\frac{2}{N},\quad \cdots,\quad a_{N-1}=\frac{1}{N}, \quad\mbox{and}\quad a_N=0
\end{equation}
and we define the sets
\begin{equation}
\begin{split}
 \label{Ljsets}
  L_1 & = \Big\{ \lambda_n \, | \,  T\leq \lambda_n<\lambda_m-\lambda_{m}^{a_1}  \Big\}, \\
  L_j & =   \Big\{\lambda_n \, | \,  \lambda_m-\lambda_{m}^{a_{j-1}} \leq \lambda_n<\lambda_m-\lambda_{m}^{a_j}  \Big\}, \text{ for } 2 \le j \le N-1 \\ 
  L_N & =  \Big\{ \lambda_n \, | \,  \lambda_m-\lambda_{m}^{a_{N-1}}\leq \lambda_n<\lambda_m-1 \Big\},\\
 L_{N+1} & =  \Big\{ \lambda_n \, | \, \lambda_m+1\leq\lambda_n<\lambda_m+\lambda_m^{a_{N-1}}  \Big\}, \\
   L_j & =   \Big\{\lambda_n \, | \,  \lambda_m+\lambda_{m}^{a_{2N+1-j}} \leq \lambda_n<\lambda_m+\lambda_{m}^{a_{2N-j}}  \Big\}, \text{ for } N+2 \le j \le 2N-1 \\ 
   L_{2N} & =  \Big\{\lambda_n \, | \, \lambda_m+\lambda_{m}^{a_1}\leq \lambda_n<2\lambda_m  \Big\},  \\
L_{2N+1} & = \Big\{\lambda_n \, | \,  2\lambda_m \leq\lambda_n  \Big\}.
  \end{split}
  \end{equation} 
Then, by \eqref{sum2} we have
\begin{equation}\label{sum2-sig}
  \Sigma_2=\sum_{\ell=1}^{2N+1} \sigma_{\ell},
\end{equation}
where 
\begin{equation}
 \label{sigmallformula}
\sigma_{\ell}=
  \sum_{T\leq \lambda_m \leq X} r_m  \sum_{\lambda_n \in L_{\ell}} \frac{r_n}{|\lambda_m-\lambda_n|}.
\end{equation}
Some of the $L_{\ell}$'s may be empty when $\lambda_m$ is close to $T$.  In this case, 
we trivially have $\sigma_{\ell}=0$.  We now estimate $\sigma_{\ell}$.  We begin with the case $\ell=1$. 
By the triangle inequality
\begin{equation}
 \label{sigma1bd}
  |\sigma_1| \le \sum_{T\leq \lambda_m\leq X}|r_m| \sum_{\lambda_n\in L_1}\frac{|r_n|}{|\lambda_m-\lambda_n|}.
 \end{equation}
By the  definition of $L_1$,  then enlarging the summation range, and then applying \eqref{weightedfirstmoment}
 \begin{equation}
  \label{innsumbd}
  \sum_{\lambda_n\in L_1}\frac{|r_n|}{|\lambda_m-\lambda_n|} \le 
  \frac{1}{\lambda_{m}^{a_1}}  \sum_{\lambda_n\in L_1} |r_n| 
  \le  \frac{1}{\lambda_{m}^{a_1}}    \sum_{0 < \lambda_n < \lambda_m} |r_n|
  \ll   \frac{1}{\lambda_{m}^{a_1}}  \cdot \lambda_{m}^{\frac{\theta-1}{2}+\epsilon}
  = \frac{1}{\lambda_{m}^{a_1-\frac{\theta-1}{2}-\epsilon}}.
 \end{equation}
 Combining \eqref{sigma1bd} and \eqref{innsumbd} we obtain 
 \begin{equation}
   \sigma_1 \ll   \sum_{T\leq \lambda_m\leq X}  \frac{|r_m|}{\lambda_{m}^{a_1-\frac{\theta-1}{2}-\epsilon}}
   =  \sum_{T\leq \lambda_m\leq X}  \frac{|r_m|}{\lambda_{m}^{1-\frac{1}{N}-\frac{\theta-1}{2}-\epsilon}}
   \ll  \sum_{T\leq \lambda_m\leq X}  \frac{|r_m|}{\lambda_{m}^{1-\frac{\theta-1}{2}-2\epsilon}}
 \end{equation}
 by \eqref{Ninequality}.  Note that $1-\frac{\theta-1}{2}-2\epsilon = a-1$ with $a= \frac{5-\theta}{2}-2 \epsilon$. 
 Observe that $a > \frac{\theta+1}{2}$ follows from the condition \eqref{epsilonrange}.  Thus, we may apply  
  Lemma \ref{pslemma} (i) with $a= \frac{5-\theta}{2}-2 \epsilon$ to 
 obtain 
 \begin{equation}
  \sigma_1 \ll \frac{1}{T^{  \frac{5-\theta}{2}-2 \epsilon-\frac{1+\theta}{2}-\epsilon}}
  = \frac{1}{T^{2 -3\epsilon}}. 
 \end{equation}
We now treat the case  $2\leq  \ell \leq 2N$.
By two applications of the Cauchy-Schwarz inequality, we have 
\begin{eqnarray}
  \sigma_{\ell}&\ll& \left(\sum_{T\leq \lambda_m\leq X} |r_m|^2\right)^{\frac{1}{2}}
  \left(\sum_{T\leq \lambda_m\leq X}\left(\sum_{\lambda_n\in L_{\ell}}\frac{|r_n|}{|\lambda_m-\lambda_n|}\right)^2\right)^{\frac{1}{2}}\nonumber\\
  &\ll&\left(\frac{1}{T^{2-\theta}}\right)^{\frac{1}{2}}\left(\sum_{T\leq \lambda_m\leq X}\left(\sum_{\lambda_n\in L_{\ell}}\frac{|r_n|^2}{|\lambda_m-\lambda_n|^2}\right) \mathcal{N}(L_{\ell})\right)^{\frac{1}{2}}\nonumber\\
\end{eqnarray}
where $\mathcal{N}(L_{\ell})$ is the number of $\lambda_n$'s in $L_{\ell}$. By the definition of the $L_j$'s in \eqref{Ljsets}
and the definition of the $a_j$'s in  \eqref{ajs} it follows that 
\begin{equation}
\mathcal{N}(L_{\ell} )\ll \lambda_m^{a_{{\ell}}}
-  \lambda_m^{a_{{\ell-1}}} \ll \lambda_m^{a_{{\ell}-1}+\epsilon}.
\end{equation}
Since $\lambda_n \in L_{\ell}$ it follows that 
$\frac{1}{|\lambda_m-\lambda_n|} \ll \frac{1}{\lambda_{m}^{a_{\ell}}}$.  
Combining these last two bounds, we obtain for $2\leq \ell \leq N$,
\begin{equation}\nonumber
\frac{\mathcal{N}(L_{\ell})}{|\lambda_m-\lambda_n|^2}\ll \frac{1}{\lambda_m^{2a_{\ell}-a_{{\ell}-1}-\epsilon}}=\frac{1}{\lambda_m^{1-\frac{{\ell}+1}{N}-\epsilon}}\ll\frac{1}{(\lambda_n)^{1-\frac{\ell+1}{N}-\epsilon}}, 
\end{equation}
where in the last step we used that  $\lambda_n\asymp\lambda_m$ for $\lambda_n\in L_{\ell}$.
Then, we have
\begin{equation}
  \sigma_{\ell}  \ll
  \frac{1}{T^{1-\frac{\theta}{2}-\epsilon}}\left(\sum_{T\leq \lambda_m\leq X}
  \sum_{\lambda_n\in L_{\ell}}
  \frac{|r_n|^2}{|\lambda_n|^{1-\frac{\ell+1}{N}-\epsilon}}\right)^{\frac{1}{2}} 
   \ll \frac{1}{T^{1-\frac{\theta}{2}-\epsilon}}\left(\sum_{T\leq \lambda_n\leq X}
    \frac{|r_n|^2}{|\lambda_n|^{1-\frac{\ell+1}{N}-\epsilon}}
  m_{\ell}(\lambda_n) \right)^{\frac{1}{2}},
\end{equation}
where we obtained the second bound by swapping  summation  and we let 
\begin{equation}
  \label{mell}
 m_{\ell}(\lambda_n)=\#\{\lambda_m: \lambda_m-\lambda_m^{a_{\ell-1}}\leq \lambda_n<\lambda_m-\lambda_m^{a_{\ell}}\}.
\end{equation}
Note that this implies for such $n$, $\lambda_n \le \lambda_m \le C \lambda_n$ for some constant $C = C(\epsilon)$. 
By  \eqref{ass4}
\begin{equation}
\begin{split}
 m_{\ell}(\lambda_n)& =\#\{\lambda_m: \lambda_n+\lambda_m^{a_{\ell}}< \lambda_m\leq\lambda_n+\lambda_m^{a_{\ell-1}}\} \\
 & \ll \#\{\lambda_m: \lambda_n+\lambda_{n}^{a_{\ell}}<\lambda_m\leq \lambda_n+(C\lambda_n)^{a_{\ell-1}} \}\ll(\lambda_n)^{1-\frac{\ell-1}{N}+\epsilon}.
 \end{split}
 \end{equation}
 Therefore for $2 \le \ell \le N$
 \begin{equation}
\begin{split}
\label{sigma-l}
 \sigma_{\ell} & \ll \frac{1}{T^{1-\frac{\theta}{2}-\epsilon}}\left(\sum_{T\leq \lambda_n\leq X}
  \frac{|r_n|^2}{|\lambda_n|^{1-\frac{\ell+1}{N}-\epsilon}} \cdot  (\lambda_n)^{1-\frac{\ell-1}{N}+\epsilon} \right)^{\frac{1}{2}} \\
  & = 
    \frac{1}{T^{1-\frac{\theta}{2}-\epsilon}}\left(\sum_{T\leq \lambda_n\leq X}
  |r_n|^2 \lambda_{n}^{\frac{2}{N}} \right)^{\frac{1}{2}} 
  \\ 
   & \le  
    \frac{1}{T^{1-\frac{\theta}{2}-\epsilon}}\left(\sum_{T\leq \lambda_n\leq X}
  |r_n|^2 \lambda_{n}^{2 \epsilon} \right)^{\frac{1}{2}} 
  \\ 
& \ll  \frac{1}{T^{1-\frac{\theta}{2}-\epsilon}} \cdot \left( \frac{1}{T^{2-\theta-2\epsilon}} \right)^{\frac{1}{2}}
= \frac{1}{T^{2 - \theta- 2 \epsilon}} 
\end{split}
\end{equation}
where we have used the bound $\frac{2}{N} \le \frac{2}{\epsilon} $ and applied Lemma \ref{pslemma} (ii) with $b=2-2 \epsilon$
(note that we have $b > \theta$ by \eqref{epsilonrange}. 
Similarly, we have
\begin{equation}\label{sig-1-2}
  \sigma_{\ell}\asymp \sigma_{2N+1-\ell}, \text{~for~} N+1\leq \ell \leq 2N.
\end{equation}
 We now treat the case $\ell=2N+1$. We have
\begin{equation}\nonumber
  \sigma_{2N+1}\ll\sum_{T\leq \lambda_m\leq X} |r_m|
  \left(\sum_{j=1}^{\infty}\sum_{2^j\lambda_m\leq\lambda_n\leq 2^{j+1}\lambda_m}\frac{|r_n|}{|\lambda_m-\lambda_n|}\right).
  \end{equation}
The inner sum is 
\begin{eqnarray*}
  &\ll& \sum_{j=1}^{\infty}\frac{1}{(2^j-1)\lambda_m}\Big(\sum_{2^j\lambda_m\leq\lambda_n\leq 2^{j+1}\lambda_m} |r_m| \Big)
  \nonumber\\
  &\le& \sum_{j=1}^{\infty}\frac{1}{(2^j-1)\lambda_m}\Big(\sum_{2^j\lambda_m\leq\lambda_n\leq 2^{j+1}\lambda_m} |r_m|^2 \Big)^{\frac{1}{2}}
  \Big(\sum_{2^j\lambda_m\leq\lambda_n\leq 2^{j+1}\lambda_m} 1 \Big)^{\frac{1}{2}}
  \nonumber\\
   &\le& \sum_{j=1}^{\infty}\frac{1}{(2^j-1)\lambda_m}\Big(
   \frac{1}{(2^j \lambda_m)^{2-\theta}}
    \Big)^{\frac{1}{2}}
  \Big( (2^{j+1}-2^j) \lambda_m \log(2^{j+1} \lambda_m)  \Big)^{\frac{1}{2}}
  \nonumber\\
 & \ll& \sum_{j=1}^{\infty}\frac{1}{(2^j-1)\lambda_m} \cdot
   \frac{1}{ 2^{j/2} \lambda_{m}^{1-\theta/2}} \cdot  2^{j/2} \lambda_{m}^{\frac{1}{2}} (\lambda_{m}^{\epsilon} + \sqrt{j})
  \nonumber\\
   & \ll& 
   \frac{\lambda_{m}^{\frac{1}{2}}}{\lambda_{m} \cdot \lambda_{m}^{1-\theta/2}}
   \sum_{j=1}^{\infty} \frac{1}{2^j} (\lambda_{m}^{\epsilon} + \sqrt{j}) \\
     \nonumber\\
& \ll & \frac{1}{\lambda_{m}^{\frac{3}{2}-\frac{\theta}{2}-\epsilon}}.
 \end{eqnarray*}
 Note that $\frac{3}{2}-\frac{\theta}{2}-\epsilon> \frac{\theta+1}{2}$ and thus by Lemma \ref{pslemma} (i)
 \begin{equation}
   \sigma_{2N+1} \ll \sum_{\lambda_m \ge T} \frac{|r_m|}{\lambda_{m}^{\frac{3}{2}-\frac{\theta}{2}-\epsilon}}
   \ll  \frac{1}{T^{(\frac{5}{2}-\frac{\theta}{2}-\epsilon)-(\frac{\theta+1}{2})-\epsilon}}
   = \frac{1}{T^{2-\theta-2 \epsilon}}. 
 \end{equation}
 By \eqref{sum2-sig} and estimates for the $\sigma_{\ell}$, it follows that 
\begin{equation}\label{Sigma2bd}
  \Sigma_2\ll_{\epsilon} \frac{1}{T^{2-\theta-4\epsilon}}.
\end{equation}
Thus combining \eqref{Sigma1bd} and \eqref{Sigma2bd}, 
we obtain \eqref{rdoublesumbd}. 
 \end{proof}
 \section{Concluding remarks}
 
 Lamzouri's method provides a lower bound for  the $\limsup$ and an upper bound for the 
 $\liminf$ in   \eqref{lblimsup}. 
 Is it possible to prove the corresponding upper bound assuming ELI (Conjecture \ref{ELI})? 
 
 Further, Lamzouri's article deals with the case that LI is true.  What happens in the case that LI fails?
 Note that the distribution of $M(e^{y})e^{-y/2}$ depends on how the ray 
 $\{ y ( \frac{\gamma_1}{2 \pi}, \cdots,  \frac{\gamma_1}{2 \pi}) \ | \ y \in \mathbb{R}  \} /\mathbb{Z}^N$  
 is distributed within the torus $\mathbb{T}^N:=
 \mathbb{R}^N/\mathbb{Z}^N$ 
 as $N \to \infty$.  If  there are linear relations among the ordinates 
 $\bm{\gamma}=(\gamma_n)_{n \in \mathbb{N}}$, then one can study the distribution 
 of $M(e^{y})e^{-y/2}$ by considering  $  {\bf X}_{M}= {\bf X}_{ss}+{\bf X}_{R}$  where ${\bf X}_{ss}$ would be the sum as in \eqref{Mrandom} 
 restricted to the self-sufficient zeros (see  \cite[Definition 1.3 , p. 3564]{MN}) and ${\bf X}_{R}$ would be the analogous sum with the remaining zeros (see Section 3.3 of \cite{MN} for a discussion of this). It is possible that a careful analysis of the random variables ${\bf X}_{ss}$ and ${\bf X}_{R}$ might lead to a variant of Conjecture  \ref{Ngconjecture} in the case that linear relations exist and a positive proportion of the ordinates $\bm{\gamma}=(\gamma_n)_{n \in \mathbb{N}}$ are self-sufficient. 

Another avenue of research would be to study sequences 
$\bm{\lambda}=(\lambda_n)_{n \in \mathbb{N}}$ and $\mathbf{r}=(r_{n})_{n \in \mathbb{N}}$   
which do not satisfy Assumption 1 \eqref{ass1}.  For instance, in Meng's work on the distribution of $k$-free numbers $\lambda_n =\gamma_n$ and $r_n = \frac{\zeta(\rho_n/k)}{\zeta'(\rho_n)} $ and thus  $\sum_{0<\lambda_n \leq T}  |r_{n}| \ll T^{\frac{1}{2}-\frac{1}{2k} +\e}$. In Meng's Conjecture \cite[p. 296, equation (1.8)]{Me} what are the values for $C_k$?  One could also study summatory functions with complex values.  What are the correct sizes of  
$\psi(x,\chi)$ and $\psi(x,\pi)$? In Section \ref{examples} the real part of this latter function was studied.  There are many other examples to which Theorem \ref{conjthm} applies and it would be interesting to further study some of them  (eg. Chebatorev's density theorem, prime number races). 
 
\section*{Acknowledgements} 
Thank-you to Youness Lamzouri and Nicol Leong for providing comments on the article.

\end{document}